\documentclass[a4paper,10pt]{amsart} 
\usepackage{amsmath,amssymb,amsfonts,graphicx,epsfig,color,url}
\usepackage{latexsym,mathrsfs,cite} 

\usepackage{float}

\usepackage{hyperref} 
\usepackage{verbatim}
\usepackage{comment}
\usepackage[title,titletoc,toc]{appendix}

\usepackage{graphicx}

\title{Dispersive shocks in Quantum Hydrodynamics with viscosity}
\author{Corrado Lattanzio, Pierangelo Marcati, Delyan Zhelyazov}
\address[Corrado Lattanzio]{DISIM, Department of Information Engineering, Computer Science and Mathematics \\ University of L'Aquila, Italy}
\email{corrado@univaq.it}
\address[Pierangelo Marcati]{GSSI, Gran Sasso Science Institute, L'Aquila, Italy}
\email{pierangelo.marcati@gssi.it}
\address[Delyan Zhelyazov]{GSSI, Gran Sasso Science Institute, L'Aquila, Italy \newline and DISIM, Department of Information Engineering, Computer Science and Mathematics \\ University of L'Aquila, Italy}\email{delyan.zhelyazov@gssi.it}

\newtheorem{lemma}{Lemma}
\theoremstyle{definition}
\newtheorem{remark}{Remark}
\theoremstyle{plain}
\newtheorem{corollary}{Corollary}

\begin{document}

\begin{abstract}
In this paper we study existence and stability of shock profiles for  a 1-D compressible Euler system in the context of Quantum Hydrodynamic models. The dispersive term is originated by the quantum effects described through the Bohm potential; moreover we introduce a  (linear) viscosity to analyze its interplay with the former while proving existence, monotonicity and stability of travelling waves connecting a Lax shock for the underlying Euler system.
The existence of monotone profiles  is proved  for sufficiently small shocks;  while the case of large shocks leads to  the (global) existence for an oscillatory profile, where dispersion plays a significant role. 
The spectral analysis of the linearized problem about a profile is also provided. In particular, we derive a sufficient condition for the stability of the essential spectrum and we estimate the maximum modulus of the eigenvalues in the unstable plane, using a careful analysis of the Evans function.
\end{abstract}

\maketitle

\section{introduction}
The aim of this paper is to study how dissipation interacts with dispersion in terms of existence and stability of traveling wave profiles, or dispersive shocks. We consider Quantum Hydrodynamics with a linear viscosity term
\begin{equation}
\label{eq_sys}
\left\{
\begin{array}{ll}
\rho_t+m_x=0,\\
m_t+\Big{(}\frac{m^2}{\rho}+\rho^\gamma\Big{)}_x=\epsilon \mu m_{xx}+\epsilon^2 k^2 \rho \Big(\frac{(\sqrt{\rho})_{xx}}{\sqrt{\rho}}\Big{)}_x,
\end{array}
\right.
\end{equation}
where $\rho=\rho(t,x)>0$, $m=m(t,x)$, $\gamma \geq 1$, $0<\epsilon\ll1$, $\mu>0$, $k>0$. Here $\epsilon \mu$ and $\epsilon^2 k^2$ are the viscosity and dispersive coefficients, respectively. Moreover $\rho^\gamma$ is the pressure and we consider $t\geq0$ and $x \in \mathbb{R}$. The  dispersive term is due to the Bohm potential and the resulting system is used for instance in superfluidity or to model semiconductor devices. 

The first studies concerning dispersive terms can be found
by \cite{Sagdeev, Gurevich}; see also \cite{Zak,Gurevich1,Nov,Hoefer}, and  \cite{Hoefer1} (and the reference therein) for a quite complete analysis via the Whitham modulation theory.
The first attempt to analyze the spectral theory of the linearized operator around dispersive shocks has been discussed in  \cite{Humpherys} regarding the case of  $p$-system with real viscosity and linear capillarity, while the mathematical theory  for  the 
Quantum Hydrodynamics   can be found in \cite{AM1, AM2, AMtf, AMDCDS,AS,Michele1,Michele, DM, DM1, DFM, GLT}.
In the present paper, we study in particular the effects of the combination of dispersion and dissipation effects in terms of existence and stability of profiles for such models. In addition, we are able to discuss  the spectrum of the linearized operator also  in the case of \emph{non monotone} shocks for our model in Eulerian coordinates; we also underline here that   a more detailed numerical description of the behavior of the Evans function close to the zero eigenvalue is presented in the companion paper \cite{LMZ18_num}.

 We first  present a \emph{local} result, concerning the study of profiles for sufficiently small shocks for the underlying Euler system, where both viscosity and capillarity terms are neglected. The existence result is proved by means of  a bifurcation argument, where the bifurcation parameter is the difference between the end states of the shock. After this quite standard result, we focus on our main interest, namely the existence and stability of profiles for \emph{large} shocks, showing in particular the combined effect of dissipation (coming from the viscosity term) and dispersion (coming from the capillarity term).
 Even if our result will require the viscosity coefficient to be ``dominant'', our results include cases when the  dispersion plays a significant role, giving rise to the existence of oscillatory profiles. This result is achieved by showing that the dynamical system solved by the profile possesses an invariant region if the dissipation is sufficiently large.

Then we focus on the study of stability properties of such profiles, also taking advantage of the numerical tests described in full details in the companion paper \cite{LMZ18_num}.
We start by investigating the spectrum of the linearized system about a profile: we analyze the essential spectrum and show that it is stable for subsonic or sonic end states. Moreover we give a bound for the real parts of the eigenvalues, by using an energy estimate, and exclude the presence of eigenvalues with non--negative real part for  $|\lambda|$ sufficiently big. This last result is not obtained only via an asymptotic result as $|\lambda|\to +\infty$, but rather we give an explicit estimate for the constant  which bounds from above the modulus of possible eigenvalues. This last more ``quantitative'' result is  fundamental to explicitly localize the region of the unstable half plane where eigenvalues may lie and then to numerically analyze such region, giving numerical evidence of spectral stability; see \cite{LMZ18_num} for further details.

The paper is organized as follows. In the Section \ref{profile_equation} we recall basic facts on the underlying Euler system to then derive the second order equation satisfied by a traveling wave profile for  \ref{eq_sys}. In Section \ref{existence_of_profiles} we present the (local and global) existence results for the profile. Finally, the   last section is devoted to the study of the linearized system about the profile and to present the results about spectral stability. 

\section{Profile equation} \label{profile_equation}
Let us start by recalling some well known facts  concerning the Euler system
\begin{equation*}
\left\{
\begin{array}{ll}
\rho_t+m_x=0,\\
m_t+\Big{(}\frac{m^2}{\rho}+\rho^\gamma\Big{)}_x=0
\end{array}
\right.
\end{equation*}
which will be used later for the analysis of the Quantum Hydrodynamics system.
The flow velocity is denoted by $u=m/\rho$. Let $U=(\rho,m)$. The eigenvalues of the Jacobian (characteristic speeds) are 
\begin{equation*}
\lambda_1(U)=u-c_s(\rho),\mbox{ }\lambda_2(U)=u+c_s(\rho),
\end{equation*}
where
\begin{equation*}
c_s(\rho)=\sqrt{\gamma \rho^{\gamma-1}}=\sqrt{\frac{d (\rho^\gamma)}{d\rho}}
\end{equation*}
denotes the sound speed. A shock wave with end states $\rho^\pm$ and $m^\pm$ and shock speed $s$ satisfies the Rankine-Hugoniot conditions:
\begin{eqnarray}
\label{Rankine_Hugoniot}
m^+-m^-=s(\rho^+-\rho^-),\\
\label{Rankine_Hugoniot2}
\Big{(}\frac{m^2}{\rho}+\rho^{\gamma}\Big{)}^+-\Big{(}\frac{m^2}{\rho}+\rho^{\gamma}\Big{)}^-=s(m^+-m^-).
\end{eqnarray}
A k-shock satisfies the Lax entropy condition:
\begin{equation*}
\lambda_k(U^+)<s<\lambda_k(U^-).
\end{equation*}

Let us consider a traveling wave profile
\begin{equation*}
\rho=P\Big{(}\frac{x-s t}{\epsilon}\Big{)}\mbox{, }m=J\Big{(}\frac{x-s t}{\epsilon}\Big{)}\mbox{. }
\end{equation*}
As customary,  the speed  $s \in \mathbb{R}$  of the travelling wave and its limiting end states 
\begin{equation*}
\lim_{y\rightarrow \pm \infty}P(y)=P^{\pm}\  \hbox{and}\ \lim_{y\rightarrow \pm \infty}J(y)=J^{\pm}
\end{equation*}
satisfy the Rankine--Hugoniot conditions \eqref{Rankine_Hugoniot}-\eqref{Rankine_Hugoniot2}.\\
We rewrite the Bohm potential in   conservative form
\begin{equation*}
\rho \Big(\frac{(\sqrt{\rho})_{xx}}{\sqrt{\rho}}\Big{)}_x=\frac{1}{2}\Big{(}\rho (\ln \rho)_{xx}\Big{)}_x\mbox{. }
\end{equation*}
After substituting the profiles $P$ and $J$ in the system \eqref{eq_sys} and multiplying by $\epsilon$ we obtain
\begin{align}
- s P'+J'&=0,\label{preq1}\\
-s J' + \Big{(}\frac{J^2}{P}+P^\gamma\Big{)}'&=\mu J''+\frac{k^2}{2}(P(\ln P)'')'\mbox{, }
\label{preq2}
\end{align}
where $'$ denotes $d/dy$ and $P=P(y)$, $J=J(y)$. 
Integrating equation \eqref{preq1}, we get 
\begin{equation*}
J(y)-sP(y)=J^--sP^-
\end{equation*}
We can also integrate \eqref{preq1} from $y$ to $+\infty$ to get
\begin{equation*}
J(y)-sP(y)=J^+-sP^+
\end{equation*}
So we obtain 
\begin{equation}
\label{equation_j}
J(y)=sP(y)-A,
\end{equation}
 with
\begin{equation*}
A=s P^{\pm}-J^{\pm},
\end{equation*}
as follows from the Rankine-Hugoniot condition \eqref{Rankine_Hugoniot}.\\
Substituting the expression for $J(y)$ into equation \eqref{preq2} and integrating we get
\begin{eqnarray*}
\int_{-\infty}^y\bigg{(} -s(sP(x)-A)' + \Big{(}\frac{(sP(x)-A)^2}{P(x)}+P(x)^\gamma\Big{)}' \bigg{)}dx \\
 = \int_{-\infty}^y \bigg{(}\mu sP(x)''+\frac{k^2}{2}(P(x)(\ln P(x))'')' \bigg{)}dx.
\end{eqnarray*}
We can also integrate from $y$ to $+\infty$. We obtain the planar ODE
\begin{equation}\label{2Dsys}
P''=\frac{2}{k^2} f(P) - \frac{2 s \mu}{k^2} P' + \frac{P'^2}{P},
\end{equation}
where
\begin{align*}
f(P)&=-s(s P-A)+\frac{(sP-A)^2}{P}+P^{\gamma}-B\nonumber\\
&=P^{\gamma}-(As+B)+\frac{A^2}{P}.
\end{align*}
Here the constant $B$ is given by
\begin{equation*}
B=-s J^{\pm}+\Big{(}\frac{J^2}{P}+P^{\gamma}\Big{)}^{\pm}.
\end{equation*}
Let $P'=Q$, we get
\begin{align}
\label{sys_ODE1}
P'=Q&=f_1,\\
Q'=\frac{2 f(P)}{k^2}-\frac{2\mu s}{k^2} Q+\frac{Q^2}{P}&=f_2.\label{sys2}
\end{align}
The constants $A,B$ in $f(P)$ can be expressed in terms of $P^{\pm}$:
\begin{align}
\label{f_roots0}
f(P)&= P^\gamma +\frac{P^-P^+}{P}\frac{(P^+)^\gamma-(P^-)^\gamma}{P^+-P^-}\nonumber\\
&\ -\frac{(P^+)^{\gamma+1}-(P^-)^{\gamma+1}}{P^+-P^-}.
\end{align}
We have
\begin{equation*}
f''(P)=\gamma (\gamma-1)P^{\gamma-2}+\frac{2 P^- P^+}{P^3}\frac{(P^+)^\gamma-(P^-)^\gamma}{P^+-P^-},
\end{equation*}
so $f''(P)>0$ as a sum of a non-negative and a positive term.\\
The function $f(P)$ has two zeros $P^{\pm}$. The system \eqref{sys_ODE1}-\eqref{sys2} has two equilibria $[P^-,0]$ and $[P^+,0]$. \\
Suppose $P^+<P^-$ and $s>0$. The jacobian evaluated at the equilibria is
\begin{equation*}
M =  \begin{bmatrix}
    0 & 1 \\
 \frac{2 f'(P^{\pm})}{k^2} & -\frac{2 \mu s}{k^2} 
  \end{bmatrix}.
\end{equation*}
Since $f''(P)>0$, we have $f'(P^+)<0$ and $f'(P^-)>0$.\\
The eigenvalues of $M$ at $P^+$ are
\begin{equation*}
\lambda_{1,2}(P^+)=\frac{-s \mu \pm \sqrt{2 k^2 f'(P^+) + s^2 \mu^2}}{k^2}.
\end{equation*}
Since $f'(P^+)<0$ we have either $2k^2f'(P^+)+s^2\mu^2<0$ or $2k^2f'(P^+)+s^2\mu^2 \geq 0$ and $\sqrt{2k^2f'(P^+)+s^2 \mu^2}<s \mu$. Hence $\Re \lambda_1(P^+), \Re \lambda_2(P^+) <0$. If $2k^2f'(P^+)+s^2\mu^2<0$ the eigenvalues have nonzero imaginary parts.\\
The eigenvalues of $M$ at $P^-$ are
\begin{equation*}
\lambda_{1,2}(P^-)=\frac{-s \mu \pm \sqrt{2 k^2 f'(P^-) + s^2 \mu^2}}{k^2}.
\end{equation*}
At $P^-$ we have $f'(P^-)>0$, so $\sqrt{2 k^2 f'(P^-)  + s^2 \mu^2}> s \mu$. Therefore $\lambda_1(P^-)>0$ and $\lambda_2(P^-)<0$.
\section{Existence of shock profiles}\label{existence_of_profiles}
Here we prove existence of shock wave profiles for the system \eqref{eq_sys}, connecting the end states $P^{\pm}$ and $J^{\pm}$, which satisfy the Rankine-Hugoniot conditions \eqref{Rankine_Hugoniot}-\eqref{Rankine_Hugoniot2}. We consider both existence of sufficiently small shocks, and possibly oscillatory profiles for large shocks, where the dispersion plays a significant role. The ratio $\mu/k$ controls how oscillatory the shocks are.

\subsection{Local existence of profiles}\label{subsec1}
In this section we consider local existence of profiles using a bifurcation theory argument about the variable $P^+ -P^- =h$. In particular, proving  a transcritical bifurcation at $h=0$, we obtain existence of profiles for small shocks without  restrictions on the coefficients $\mu$ and $k$.

We start by rewriting $f$ in \eqref{f_roots0} in the variables $u=P-P^-$ and $P^+ -P^- =h$ as follows:
\begin{align*}
\tilde{f}(u,h)&=(P^-+u)^\gamma +\frac{P^-(P^-+h)}{P^-+u}\frac{(P^-+h)^\gamma-(P^-)^\gamma}{h}\nonumber\\
&\ -\frac{(P^-+h)^{\gamma+1}-(P^-)^{\gamma+1}}{h}.
\end{align*}
Let $u'=v$. The system \eqref{sys_ODE1}-\eqref{sys2} becomes
\begin{eqnarray}
\label{sys_ODE1a}
u'=v=\tilde{f}_1(u,v,h),\\
\label{sys_ODE1a2}
v'=\frac{2}{k^2}\tilde{f}(u,h)-\frac{2 s \mu}{k^2}v+\frac{v^2}{P^-+u}=\tilde{f}_2(u,v,h).
\end{eqnarray}
In the following lemma we transform the system \eqref{sys_ODE1}-\eqref{sys2} to a normal form, which is given by  two scalar decoupled equations, for which we can prove the existence of a heteroclinic connection.
\begin{lemma}\label{lemma_loc_existence}
For any $P^->0$ there is an $\epsilon>0$, such that if $s>0$ ($s<0$), $|P^+-P^-|<\epsilon$ and $P^+<P^-$ ($P^-<P^+$), there exists a heteroclinic for the system \eqref{sys_ODE1}-\eqref{sys2}, connecting $[P^-,0]$ to $[P^+,0]$.
\end{lemma}
\begin{proof}
Suppose $h< 0$, $s>0$. Let $A(u,0,h)= D(\tilde{f}_1,\tilde{f}_2)/D(u,v)\vert_{(v=0)}$ be the jacobian evaluated for $v=0$,  $h< 0$, and any $u$, that is
\begin{equation*}
A(u,0,h)=  \begin{bmatrix}
    0 & 1 \\
\frac{2}{k^2} \frac{\partial \tilde{f}(u,h)}{\partial u} & -\frac{2 \mu s}{k^2} 
  \end{bmatrix}.
\end{equation*}
Then, system \eqref{sys_ODE1a}--\eqref{sys_ODE1a2} admits two equilibria $[0,0]$ and $[h,0]$, corresponding to $[P^-,0]$ and $[P^+,0]$.
The calculation in the end of Section \ref{profile_equation} shows that $A(0,0,h)$ has two eigenvalues with negative real parts, and $A(h,0,h)$ has a positive and a negative eigenvalue. 
Therefore, the equilibrium $[0,0]$ is a saddle and $[h,0]$ is stable.

Now, in order to study this system close to the bifurcation value $h=0$, 
we  extend $\tilde{f}$ and its derivatives by continuity:
\begin{align}\label{eq:firstline}
\tilde{f}(0,0)&=0,\mbox{ }\frac{\partial \tilde{f}(0,0)}{\partial u}=0,\mbox{ }\frac{\partial \tilde{f}(0,0)}{\partial h}=0,\\
\nonumber
\frac{\partial^2 \tilde{f}(0,0)}{\partial u \partial h}&=-\frac{\gamma(\gamma+1)}{2}(P^{-})^{\gamma-2}<0,\\
\nonumber
\frac{\partial^2 \tilde{f}(0,0)}{\partial u^2}&=\gamma(\gamma+1)(P^{-})^{\gamma-2}>0.
\end{align}
As a consequence, we get
\begin{equation*}
A=A(0,0,0)=  \begin{bmatrix}
    0 & 1 \\
    0 & -\frac{2 \mu s}{k^2}
  \end{bmatrix},
  \end{equation*}
and clearly its eigenvalues   are $\lambda_1=0$, $\lambda_2 = -2 s \mu /(k^2)$.
Let $K$ be the matrix of column eigenvectors of $A$, namely
\begin{equation*}
K=  \begin{bmatrix}
    1 & -\frac{k^2}{2 \mu s} \\
    0 & 1
  \end{bmatrix}.
  \end{equation*}
 We change the variables $(u,v)=K(w_1,w_2)$, that is we transform the original system according to the eigenbasis at $[0,0,0]$:
\begin{eqnarray}
\label{2D_system_eigenbasis}
\begin{bmatrix}
w_1\\
w_2
\end{bmatrix}'=
\begin{bmatrix}
0 & 0\\
0 & -\frac{2 s \mu}{k^2}
\end{bmatrix}
\begin{bmatrix}
w_1\\
w_2
\end{bmatrix}
+
\begin{bmatrix}
\frac{k^2}{2 s \mu} g\\
g
\end{bmatrix},
\end{eqnarray}
where
\begin{eqnarray*}
\nonumber
g=g\big{(}w_1-\frac{k^2}{2 s \mu}w_2, w_2,h\big{)},\\
g(u,v,h)=\frac{2}{k^2}\tilde{f}(u,h)+\frac{v^2}{P^-+u}.
\end{eqnarray*}
In view of \eqref{eq:firstline}, we have
\begin{equation*}
g(0,0,0)=0,\mbox{ }\partial_u g(0,0,0)=0,\mbox{ }\partial_v g(0,0,0)=0,
\end{equation*}
so that \eqref{2D_system_eigenbasis} is of the form of a linear part and a perturbation.\\
There is a center manifold $w_2=\psi(w_1,h)$, which satisfies the tangency conditions:
\begin{eqnarray}
\nonumber
\psi(0,0)=0,\mbox{ }\frac{\partial\psi(0,0)}{\partial w_1}=0,\mbox{ }\frac{\partial\psi(0,0)}{\partial h}=0.
\end{eqnarray}
We perform center manifold reduction:
\begin{equation}\label{eq:centerred}
w_1'=\zeta(w_1,h),
\end{equation}
where the function $\zeta(w_1,h)$ is given by the following expression:
\begin{equation*}
\zeta(w_1,h): =\frac{k^2}{2 s \mu} g\big{(}w_1-\frac{k^2}{2 s \mu}\psi(w_1,h),\psi(w_1,h),h\big{)}.
\end{equation*}
Then we have
\begin{align*}
\nonumber
\zeta(0,0)&=\frac{k^2}{2 s \mu} g(0,0,0)=0,\mbox{ }\frac{\partial \zeta(0,0)}{\partial w_1}=0,\\
\frac{\partial \zeta(0,0)}{\partial h}&=0,\mbox{ } a=\frac{\partial^2 \zeta(0,0)}{\partial h \partial w_1},
\mbox{ }2b=\frac{\partial^2 \zeta(0,0)}{\partial w_1^2}.
\end{align*}
We get the normal form of a transcritical bifurcation
\begin{equation}
\label{nf_transcritical}
w_1' = a h w_1 + b w_1^2,
\end{equation}
where
\begin{equation*}
a =-\frac{\gamma(\gamma+1)}{2 s \mu} (P^-)^{\gamma-2},\mbox{ }
b= -a.
\end{equation*}
Indeed, Theorem 5.4 from \cite{Kuznetsov}, p.\ 159 implies that the system \eqref{sys_ODE1a}-\eqref{sys_ODE1a2} is locally topologically equivalent to the scalar nonlinear ODE \eqref{eq:centerred} for $w_1$ augmented with the linear equation $w_2'=sgn(\lambda_2) w_2 = - w_2$.
The nondegenerancy condition $a,b\neq 0$ is satisfied, so the parameter $h$ unfolds the bifurcation with normal form given in \eqref{nf_transcritical}. Since $\lambda_2<0$, the center manifold is attracting. We have $a<0$, $b>0$, so the trivial equilibrium $w_1=0$ is unstable and the nontrivial equilibrium $h<0$ is stable for \eqref{nf_transcritical}.
Correspondingly, \eqref{eq:centerred} has the unstable (trivial) equilibrium $w_1=0$, a stable (nontrivial) equilibrium $w_1^0(h)$ and, as a consequence, 
there exists a heteroclinic connecting $w_1$ to $w_1^0(h)$. The homeomorphism preserves the number of eigenvalues with positive (negative) real parts of the equilibria. The decoupled system with first equation \eqref{eq:centerred} has a trivial equilibrium $[0,0]$ with one positive and one negative eigenvalue, and a non-trivial equilibrium $[w_1^0(h),0]$ with two negative eigenvalues. So the equilibrium $[0,0]$ is mapped to $[P^-,0]$ and  $[w_1^0(h),0]$ is mapped to $[P^+,0]$. The heteroclinic, connecting $[0,0]$ to $[w_1^0(h),0]$, corresponds to a heteroclinic, connecting $[P^-,0]$ to $[P^+,0]$.\\
Now let $s<0$. We have $\lambda_2>0$, so the center manifold is not attracting. We have $a>0$ and $b<0$. Let $P^+>P^-$ (that is $h>0$). The equilibrium $w_1=0$ is unstable, and $w_1^0(h)=h>0$ is stable. The decoupled system has a heteroclinic, connecting $[0,0]$ to $[w_1^0(h),0]$. The equilibrium $[0,0]$ has two positive eigenvalues, and it is  mapped to $[P^-,0]$. The equilibrium $[w_1^0(h),0]$ has one positive and one negative eigenvalue, and it is mapped to $[P^+,0]$. The heteroclinic of the decoupled system corresponds to a heteroclinic, connecting $[P^-,0]$ to $[P^+,0]$.
\end{proof}
The heteroclinic orbit for $P$ constructed  in Lemma \ref{lemma_loc_existence} gives a traveling wave profile for our model, by using $J(y)$ from  equation \eqref{equation_j}.
The following corollary phrases this existence result of profiles in terms of Rankine-Hugoniot  and   Lax entropy conditions of end states, as well as super- or sub-sonic conditions.
\begin{corollary}
For any $P^->0$, there is an $\epsilon>0$, such that if $P^+>0$, $|P^+-P^-|<\epsilon$, and the end states $P^{\pm}$, $J^{\pm}$ and the speed $s$ satisfy the Lax condition for a 2-shock  with supersonic right state and $u^+>c_s(P^+)$ or a 1-shock with a subsonic left state, then there is a traveling wave profile.
\end{corollary}
\begin{proof}
Let $P^{\pm},J^{\pm},s$ satisfy the Rankine-Hugoniot conditions. We are going to express the end states $J^{\pm}$ in terms of $P^{\pm}$ and $s$, using the Rankine-Hugoniot conditions. In this way we obtain an equivalent expression.\\
Expressing $J^+$ from \eqref{Rankine_Hugoniot} we get
\begin{equation}
\label{expr_flux_rh}
J^+=J^-+s(P^+-P^-).
\end{equation}
Substituting $J^+$ into \eqref{Rankine_Hugoniot2} and dividing by the quadratic coefficient
\begin{equation*}
\frac{P^--P^+}{P^-P^+}\neq 0,
\end{equation*}
we get the quadratic equation
\begin{equation}
\label{quad_equation}
(J^-)^2-2sP^- J^-+\frac{P^-P^+\big{(}(P^+)^\gamma-(P^-)^\gamma-P^-s^2 +((P^-)^2s^2)/(P^+)\big{)}}{P^--P^+}=0.
\end{equation}
The quadratic equation \eqref{quad_equation}  has two solutions $J^{-}_{1,2}=s P^{-}-A_{1,2}$, where
\begin{equation*}
A_{1,2}=\mp \sqrt{P^+ P^-}\sqrt{\frac{(P^+)^{\gamma}-(P^-)^{\gamma}}{P^+-P^-}}
\end{equation*}
Substituting these solutions in equation \eqref{expr_flux_rh} yields two solutions $J^{+}_{1,2}=s P^{+}-A_{1,2}$.\\
\\
Suppose the shock satisfies the Lax condition
\begin{equation*}
\lambda_2(U^+)<s<\lambda_2(U^-).
\end{equation*}
Then $\lambda_2(U^+)<s$, so $u^{+}+c_s(P^+)<s$. Suppose $J^-=s P^{-}-A_1$, with $A_1<0$. This implies $J^+=s P^{+}-A_1$. Then using $u^+=J^+/P^+$, we get $-\frac{A_1}{P^+}+c_s(P^+)<0$. Since $c_s(P^+)>0$, we get a contradiction. So $J^-=s P^{-}-A_2$, where $A_2>0$, which implies $J^+=s P^{+}-A_2$.\\
From the Lax condition we get $\lambda_2(U^+)<\lambda_2(U^-)$, hence $u^{+}+c_s(P^+)<u^{-}+c_s(P^-)$. Using $u^{\pm}=J^{\pm}/P^{\pm}$, we get
\begin{equation*}
\frac{A_2(P^{+}-P^{-})}{P^{+} P^-}<c_s(P^-)-c_s(P^+).
\end{equation*}
Now, suppose $P^+>P^-$. Since the sound speed $c_s(P)$ is nondecreasing, we get $c_s(P^-)-c_s(P^+) \leq 0$. On the other hand $(A_2(P^{+}-P^{-}))/(P^{+} P^-)>0$. So we get a contradiction. Hence $P^{+}<P^{-}$.\\
Suppose we have a 2-shock with a supersonic right state and $u^+>c_s(P^+)$. Then the Lax condition implies $s>0$. Therefore, we have the condition of Lemma \ref{lemma_loc_existence} for a local existence of a profile.\\
\\
Suppose we have a 1-shock:
\begin{equation*}
\lambda_1(U^+)<s<\lambda_1(U^-).
\end{equation*}
In this case $s<\lambda_1(U^-)$. Suppose $J^-=s P^{-}-A_2$, with $A_2>0$. This implies $J^+=s P^{+}-A_2$. Then using $u^{-}=J^{-}/P^{-}$ we get $A_2/P^-<-c_s(P^-)$. Since $c_s(P^{-})>0$, this is a contradiction. Hence $J^-=s P^{-}-A_1$, where $A_1<0$, which implies $J^+=s P^{+}-A_1$.\\
We have from the Lax condition $\lambda_1(U^+)<\lambda_1(U^-)$, so $u^+-u^-<c_s(P^{+})-c_s(P^{-})$. Using $u^{\pm}=J^{\pm}/P^{\pm}$, we get
\begin{equation*}
\frac{A_1(P^{+}-P^{-})}{P^{+} P^-}<c_s(P^+)-c_s(P^-).
\end{equation*}
Suppose $P^+<P^-$. The sound speed $c_s(P)$ is nondecreasing, and
\begin{equation*}
\frac{A_1(P^{+}-P^{-})}{P^{+} P^-}>0,
\end{equation*}
so we get a contradiction. Hence we have $P^+>P^-$.\\
Suppose we have a 1-shock with a subsonic left state. Then $|u^-|<c_s(P^{-})$, and in particular $u^-<c_s(P^{-})$. The Lax condition implies $s<0$. Hence, we have the second condition of Lemma \ref{lemma_loc_existence} for a local existence of a shock profile.
\end{proof}

\subsection{Global existence of profiles}
This section concerns the proof of existence of the profiles in the case of large amplitude shocks. In contrast to the case of local existence, here we shall need conditions between  viscosity and dispersion coefficients, and in particular the latter needs to be sufficiently strong, see 
  Lemma \ref{lemma_global_existence}.
  
We start by rescaling  the velocity $Q$ in terms of $k$ as $\tilde{Q}=k^2 P'/2$. Then the system can be rewritten as follows: 
\begin{align}\label{2D_sys_tilde}
P'&=\frac{2}{k^2}\tilde{Q}=:\tilde{f}_1,\\
\label{2D_sys_tilde1}
\tilde{Q}'&=f(P)-\frac{2 s \mu}{k^2} \tilde{Q}+\frac{2}{k^2}\frac{\tilde{Q}^2}{P}=:\tilde{f}_2.
\end{align}
The crucial observation is that the   reduced (indeed conservative) system 
\begin{eqnarray}\label{2D_sys_ham1}
P'=\frac{2}{k^2}\tilde{Q},\\
\label{2D_sys_ham2}
\tilde{Q}'=f(P)
\end{eqnarray}
admits a   homoclinic loop, which confines the heteroclinic connection  we are looking for. 
 More specifically, the homoclinic loop turns out to be the boundary of an invariant region for the full system.
 Finally, to exclude closed trajectories  inside such invariant region, we shall use the Poincar\'e--Bendixon criterion. 
 
 We start with the case $s>0$ and $P^+<P^-$.
This condition is implied by the condition for a 2-shock with a supersonic right state, and $u^+>c_s(P^+)$.\\
We have
\begin{equation*}
P^2 f'(P)=\gamma P^{\gamma+1}-A^2=0
\end{equation*}
has a unique solution $P_0=(A^2/\gamma)^\frac{1}{\gamma+1}$. Since $f'(P)=\gamma P^{\gamma-1}-P^{-2}A^2$, for $P \in (0,P_0)$, $P^{\gamma+1}<\frac{A^2}{\gamma}$, so $f'(P)<0$. Moreover $f'(P)>0$ for $P>P_0$. We are going to consider choices of parameters, for which $f(P)=0$ has one or two positive solutions. If there are two solutions, they will be the limiting values of the traveling wave profile $P^+<P^-$. We have $P^+< P_0<P^-$. Also we have 
$f''(P)>0$.\\

Let us now consider  the related  reduced system \eqref{2D_sys_ham1}-\eqref{2D_sys_ham2}, or its second order counterpart: 
\begin{equation} 
\label{ham_sys}
\frac{k^2}{2}P''=f(P),
\end{equation}
where we truncate the last two terms in \eqref{2D_sys_tilde1}. 
Let
\begin{equation*}
F(P)=\int f(z) dz = \frac{P^{\gamma+1}}{\gamma+1}-(As+B)P+A^2 \ln P,
\end{equation*}
The system \eqref{ham_sys} has energy
\begin{equation}
\label{ham_energy}
H(P,\tilde{Q})=F(P)-\frac{\tilde{Q}^2}{k^2}-F(P^-).
\end{equation}
Since $F(P)-F(P^-)=\int_{P^-}^P f(z)dz$ and $f(P)<0$ for $P^+<P<P^-$, we have $\int_{P^-}^P f(z)dz>0$ for $P^+<P<P^-$. Also $f(P)>0$ for $0<P<P^+$, and $\lim_{P \downarrow 0}F(P)=-\infty$, so there is a point $P^\star \in (0,P^+)$, such that $F(P)-F(P^-)>0$ for $P^\star<P<P^-$, and $F(P)-F(P^-)<0$ for $0<P<P^\star$. So we have $H(P,0)<0$ for $0<P<P^\star$, and $H(P,0)>0$ for $P^\star<P<P^-$.
The system \eqref{ham_sys} has a homoclinic loop starting at $P^-$, which corresponds to $H(P,\tilde{Q})=0$. Moreover the energy levels, contained inside the homoclinic loop, are compact and correspond to $H(P,\tilde{Q})=C>0$.\\
Let
\begin{equation}
\label{fun_1}
F_1(P)=F(P)-\Big{(}\frac{s \mu P}{k}\Big{)}^2-F(P^-),
\end{equation}
be the restriction of $H$ on the line $\tilde{Q}=s \mu P$.\\
Now, let $s<0$ and $P^-<P^+$. The function $f(P)$ has two roots $P^\pm$. We are going to consider the reverse parameter $\tilde{y}=-y$ and let $\xi$, $\eta$ correspond to $P$, $-\tilde{Q}$, respectively.  We are going to prove existence of a heteroclinic $[\xi,\eta]$, connecting $[P^+,0]$ to $[P^-,0]$, which corresponds to a heteroclinic, connecting $[P^-,0]$ to $[P^+,0]$. 
We rewrite the system \eqref{sys_ODE1}-\eqref{sys2} in terms of $\tilde{y}$:
\begin{align}\label{2D_sys_tilde_new_1}
\xi'&=\frac{2}{k^2}\eta,\\
\label{2D_sys_tilde_new_2}
\eta'&=f(\xi)-\frac{2 \tilde{s} \mu}{k^2} \eta+\frac{2}{k^2}\frac{\eta^2}{\xi},
\end{align}
where $\tilde{s}=-s$, $\tilde{s}>0$.
In the new variables the reduced system \ref{ham_sys} becomes
\begin{align}
\label{2D_sys_ham_new_var1}
\xi'&=\frac{2}{k^2}\eta,\\
\label{2D_sys_ham_new_var2}
\eta'&=f(\xi),
\end{align}
with $'=d/d\tilde{y}$.\\
The system \eqref{2D_sys_ham_new_var1}-\eqref{2D_sys_ham_new_var2} has energy
\begin{equation*}
\tilde{H}(\xi,\eta)=F(\xi)-\frac{\eta^2}{k^2}-F(P^+).
\end{equation*}
The system \eqref{2D_sys_ham_new_var1}-\eqref{2D_sys_ham_new_var2} has a homoclinic loop, starting at $[P^+,0]$, contained in the set $\tilde{H}(\xi,\eta)=0$.\\
Let
\begin{equation*}
\tilde{F}_1(\xi)=F(\xi)-\Big{(}\frac{s \mu \xi}{k}\Big{)}^2-F(P^+).
\end{equation*}
\begin{lemma}\label{lemma_global_existence}Let $s>0$ and $P^+<P^-$. If $F_1(P)\leq 0$ for $0<P<P^+$, then there is a traveling wave profile for the system \eqref{2D_sys_tilde}-\eqref{2D_sys_tilde1}, connecting $[P^-,0]$ to $[P^+,0]$. If in addition
\begin{equation}
\label{cond_0_0}
\frac{s \mu}{k}<\sqrt{-2 f'(P^+)},
\end{equation}
then the traveling wave profile is non-monotone.\\
Let $s<0$ and $P^-<P^+$. If $\tilde{F}_1(\xi)\leq 0$ for $0<\xi<P^-$, then there is a traveling wave profile for the system \eqref{2D_sys_tilde}-\eqref{2D_sys_tilde1}, connecting $[P^-,0]$ to $[P^+,0]$. If in addition
\begin{equation*}
-\frac{s \mu}{k}<\sqrt{-2 f'(P^-)},
\end{equation*}
then the traveling wave profile is non-monotone.
\end{lemma}
\begin{proof}
We are going to show that the homoclinic loop of \eqref{2D_sys_ham1}-\eqref{2D_sys_ham2} is confining, the orbit, contained in the unstable manifold of the saddle $[P^-,0]$ is inside it, and its $\omega$-limit set is $[P^+,0]$ (see Figure \ref{fig_heteroclinic}).\\
Let $[P(y),\tilde{Q}(y)]$ be a solution of \eqref{2D_sys_tilde}-\eqref{2D_sys_tilde1} and $\mathcal{H}(y)=H(P(y),\tilde{Q}(y))$. Then
\begin{equation*}
\mathcal{H}' = \frac{\partial H}{\partial P}P'+\frac{\partial H}{\partial \tilde{Q}}\tilde{Q}'=\frac{4 \tilde{Q}^2}{k^4}(s \mu - \frac{\tilde{Q}}{P}).
\end{equation*}
If $\mathcal{H}' \geq 0$ for any point of the homoclinic loop, then it is confining, that is  a trajectory that starts in the homoclinic loop will stay inside it for all $y \geq 0$. Also $s \mu P - \tilde{Q} \geq 0$ implies $\mathcal{H}' \geq 0$ . We require that the homoclinic loop is contained in the region $s \mu P \geq \tilde{Q}$. We would like to prove that the trajectory, contained in the unstable manifold of the saddle $[P^-,0]$ will converge to $[P^+,0]$ as $y \rightarrow +\infty$.\\
\begin{figure}[H]
\begin{center}
\includegraphics[scale=0.8]{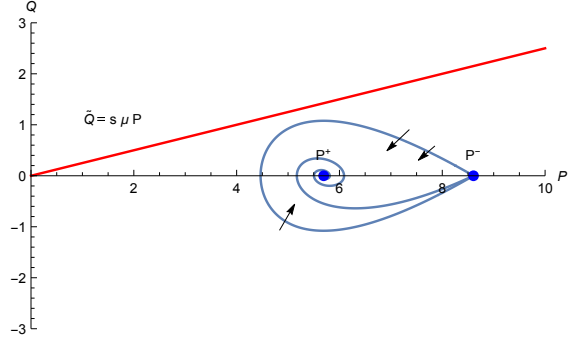}
\end{center}
\caption{The homoclimic loop and the heteroclinic connection for parameters $A=1$, $B=7.3$, $s=1$, $\gamma=1$, $\mu=0.25$, $k=\sqrt{2}$}
\label{fig_heteroclinic}
\end{figure}
We can show that the $\omega$-limit set of the trajectory, contained in the unstable manifold of the right equilibrium is the left equilibrium, if the left equilibrium is stable and we can exclude loops. For this we can use the Poincar\'e-Bendixon criterion, which states if $f \in C^1(E)$, where $E$ is a simply connected region in $\mathbb{R}^2$, and if there exists a function $B \in C^1(E)$, such that the divergence of the vector field $B f$, $\nabla \cdot (B f)$ is not identically zero and does not change sign in $E$, then the planar system $x'=f(x)$ has no closed orbits, lying entirely in $E$ (see \cite{Perko}, p.265, Theorem 2). Under the conditions of this criterion, there are no separatrix cycles or graphics of $x'=f(x)$, lying entirely in $E$ (see \cite{Perko}, p.265).\\
The divergence of $\tilde{f}/P$ is 
\begin{equation*}
\frac{\partial}{\partial P}\Big{(}\frac{\tilde{f}_1}{P}\Big{)}+\frac{\partial}{\partial Q}\Big{(}\frac{\tilde{f}_2}{P}\Big{)}=\frac{2}{k^2 P^2}(\tilde{Q}-s \mu P),
\end{equation*}
therefore it does not change sign and is not identically zero in the same region, where $H' \geq 0$.\\
Now let us consider the jacobians at the equlibria. Let $a=\frac{2}{k^2}>0$, $b=\frac{2 s \mu}{k^2}>0$. Let
\begin{equation}
\label{jacs}
 J_1=  \begin{bmatrix}
    0 & a \\
    f'(P^\pm) & 0
    \end{bmatrix},
     J_2=  \begin{bmatrix}
    0 & a \\
    f'(P^\pm) & -b
    \end{bmatrix}.
\end{equation}
Here $J_1$ is the linearization of \eqref{ham_sys} at $P^\pm$, $\tilde{Q}=0$ and $J_2$ is the linearization of \eqref{2D_sys_tilde} at $P^\pm$, $\tilde{Q}=0$. The eigenvalues of $J_2$ are $(-b \pm \sqrt{b^2+4 a f'(P^\pm)})/2$. At $P^+$, since $f'(P^+)<0$, we have either $b^2+4 a f'(P^+)<0$, or $b^2+4 a f'(P^+)\geq 0$ and $\sqrt{b^2+4 a f'(P^+)}<b$. So in any case $\Re \lambda_1, \Re \lambda_2 <0$. If $b^2+4 a f'(P^+)<0$ the eigenvalues have nonzero imaginary parts.\\
Now we will show that the vector, which is tangent to the unstable manifold of the saddle $[P^-,0]$ is directed inside the homoclinic loop. At $P^-$ $f'(P^-)>0$, so $\sqrt{b^2+4 a f'(P^-)}>b$. Therefore we have a positive and a negative eigenvalue - a saddle. The eigenvector, corresponding to $(-b + \sqrt{b^2+4 a f'(P^-)})/2$ is 
\begin{equation*}
v_2=  -\begin{bmatrix}
    \frac{(b + \sqrt{b^2+4 a f'(P^-)})}{2 f'(P^-)}\\
    1
    \end{bmatrix}
\end{equation*}
The eigenvector, corresponding to the positive eigenvalue of $J_1$ is 
\begin{equation*}
\tilde{v}_2= -\begin{bmatrix}
   \sqrt{ \frac{a}{f'(P^-)}}\\
    1
    \end{bmatrix}
\end{equation*}
If $\tilde{v}_{2,1}>v_{2,1}$, then the eigenvector $v_2$, which is tangent to the unstable subspace is pointing inside the homoclinic loop $H(P,\tilde{Q})=0$. This is true, because $2 b^2+2b\sqrt{b^2+4af'(P^-)}>0$, which implies $(b+\sqrt{b^2+4af'(P^-)})^2>\frac{a}{f'} (2 f'(P^-))^2$ and $\tilde{v}_{2,1}>v_{2,1}$ follows by taking a square root.
\vspace{5mm}\\
Now we would like to consider the sufficient condition for the existence of an oscillatory profile. From \eqref{jacs} the jacobian at the equilibrium $[P^+,0]$ has imaginary eigenvalues if and only if \eqref{cond_0_0} holds.\\
Let us express $\tilde{Q}$ as a function of $P$ from \eqref{ham_energy} and take the positive branch:
\begin{equation*}
\tilde{Q}=g(P)=k\sqrt{F(P)-F(P^-)}.
\end{equation*}
Its derivative is
\begin{equation*}
g'(P)=\frac{k}{2}\frac{f(P)}{\sqrt{F(P)-F(P^-)}}.
\end{equation*}
So $g'(P)$ vanishes if and only if $f(P)=0$. The positive branch of the homoclinic has a maximum at $P=P^+$. So if the line $\tilde{Q}=s \mu P$ does not intersect the homoclinic loop in the interval $(0,P^+]$, it will not intersect it also in the interval $(0,P^-)$. If for all $P$ in the interval $(0,P^+]$ the expression \eqref{fun_1}, which the restricion of $H$ on the line $\tilde{Q}=s \mu P$ is nonpositive, then the line $\tilde{Q}=s \mu P$ does not intersect the homoclinic loop (since $H(P,\tilde{Q})>0$ inside the loop).\\
Now, let use consider the case $s<0$.
The system \eqref{2D_sys_tilde_new_1}-\eqref{2D_sys_tilde_new_2} has the same form as \eqref{2D_sys_tilde}-\eqref{2D_sys_tilde1} and the above proof applies to it. So it has a heteroclinic, connecting $[P^+,0]$ to $[P^-,0]$, which corresponds to a heteroclinic, connecting $[P^-,0]$ to $[P^+,0]$ in the parameter $y$.
\end{proof}
We can show numerically that Lemma \ref{lemma_global_existence} applies to the parameters
\begin{equation}
\label{parameters}
A=1,\mbox{ } B=1.1,\mbox{ } s=1,\mbox{ } \gamma=3/2,\mbox{ } \mu=1,\mbox{ } k=\sqrt{2}.
\end{equation}
 These parameters correspond to a non-monotone profile (see \cite{LMZ18_num}).
\begin{remark} The minimum value of $\mu$ for a given $k$, for which we can guarantee a heteroclinic is  the value, for which $(s \mu P)^2=k^2(F(P)-F(P^-))$ has a unique solution. This is the maximum $\mu$, for which this equation has a solution and then the tangency condition is
\begin{equation*}
\label{tang}
\frac{k f(P)}{2\sqrt{F(P)-F(P^-)}}=s \mu.
\end{equation*}
\end{remark}
\begin{remark} For $\gamma=1$, the condition of Lemma \ref{lemma_global_existence} can be verified analytically. The derivative of $F_1(P)$ is 
\begin{equation*}
F_1'(P)=f(P)-2\Big{(}\frac{s \mu}{k}\Big{)}^2 P.
\end{equation*}
We have $F_1''(P)=f'(P)-2\Big{(}\frac{s \mu}{k}\Big{)}^2<0$ for $0<P<P^+$, so $F_1'(P)$ is monotonically decreasing for $0<P<P^+$. Also $\lim_{P \downarrow 0} F_1'(P)=+\infty$, $F_1'(P^+)=-2\Big{(}\frac{s \mu}{k}\Big{)}^2<0$, hence $F_1'(P)$ has one zero in the interval $(0,P^+)$.\\
If $\gamma \in \mathbb{N}$, then $P(f(P)-2(\frac{s \mu}{k})^2P)$ is a polynomial.
For $\gamma = 1$, let
\begin{equation*}
D=(As+B)^2k^4-4A^2k^2(k^2-2s^2\mu^2).
\end{equation*} 
The roots of $P(f(P)-2(\frac{s \mu}{k})^2P)=0$ are
\begin{equation*}
\frac{(As+B)k^2\pm\sqrt{D}}{2(k^2-2s^2\mu^2)}
\end{equation*}
Consider the zero of $F_1'(P)$ in the interval $(0,P^+)$. At this zero $F_1$ will have a maximum. The condition of Lemma \ref{lemma_global_existence} will be verified, if $F_1(P)\leq 0$ at the zero. This condition in conjunction with \eqref{cond_0_0} guarantees that there is an oscaillatory profle. 	We can write similar formulas for $\gamma = 2,3$.\\
\end{remark}
\begin{remark} The vector, tangent to the unstable manifold of the saddle $[P^-,0]$ is directed inside the homoclimic loop, correspoding to $H(P,\tilde{Q})=0$. So for any trajectory $[P(y),\tilde{Q}(y)]$, contained in the unstable manifold there will be $y_0$, such that $H(P(y_0),\tilde{Q}(y_0))> \epsilon_0>0$. The forward trajectory will be contained in the energy level $H(P,\tilde{Q})\geq \epsilon_0$, which contains the only steady-state $[P^+,0]$.
\end{remark}
\begin{corollary}For fixed values of $\gamma$ ($\gamma \in \{1,3/2,5/3\}$) there are non-empty intervals for $\mu$, for which a oscillatory heteroclinic exists.\end{corollary}

\section{Linearization and stability results}
\label{section_linearization}

Using the change of variables $\tau=t/\epsilon$, $y=(x-st)/\epsilon$, $u(x,t)=\bar{u}(\tau,y)$, we get the full linearized operator around the profile for \eqref{eq_sys}:
\begin{equation}
\label{operator_L}
L
\begin{bmatrix}
\tilde{\rho}\\
\tilde{J}
\end{bmatrix}
=\begin{bmatrix}
    s \tilde{\rho}_y - \tilde{J}_y\\
    s \tilde{J}_y +(\frac{J^2}{P^2}\tilde{\rho})_y-(\frac{2 J}{P}\tilde{J})_y-\gamma (P^{\gamma-1}\tilde{\rho})_y+\mu \tilde{J}_{yy}+L_V\tilde{\rho}
    \end{bmatrix},
\end{equation}
where
\begin{equation*}
L_V \tilde{\rho} = \frac{k^2}{2}\tilde{\rho}_{yyy}-2 k^2 \Big{(} (\sqrt{P})_y\Big{(}\frac{\tilde{\rho}}{\sqrt{P}}\Big{)}_y\Big{)}_y,
\end{equation*}
with associated eigenvalue problem given by
\begin{equation}
\label{eq_variable_coeff}
\lambda \begin{bmatrix}
\tilde{\rho}\\
\tilde{J}
\end{bmatrix} = L\begin{bmatrix}
\tilde{\rho}\\
\tilde{J}
\end{bmatrix}.
\end{equation}

A related constant coefficient linear operator is clearly obtained linearizing our original system about a constant state, thus obtaining the same operator, but for $s=0$.
 Denote
\begin{equation}\label{eq:constalbeta}
\alpha=\frac{\bar J^2}{\bar P^2}-\gamma \bar P^{\gamma-1}\mbox{ , }\beta= -\frac{2 \bar J}{\bar P}.
\end{equation}
Then the operator, corresponding to the linearization around the constant steady--state $(\bar P, \bar J)$ is
\begin{eqnarray*}
L_c
\begin{bmatrix}
\tilde{\rho}\\
\tilde{J}
\end{bmatrix}
=\begin{bmatrix}
     - \tilde{J}'\\
    \alpha \tilde{\rho}'+\beta \tilde{J}'+\mu \tilde{J}''+\frac{k^2}{2}\tilde{\rho}'''
    \end{bmatrix},
\end{eqnarray*}
where $'=d/dy$.
The asymptotic operators  at $\pm\infty$ for \eqref{operator_L}  are given by
\begin{eqnarray*}
L_{\pm\infty}
\begin{bmatrix}
\tilde{\rho}\\
\tilde{J}
\end{bmatrix}
=\begin{bmatrix}
    s \tilde{\rho}' - \tilde{J}'\\
    \alpha^\pm \tilde{\rho}'+\beta^\pm \tilde{J}'+\mu \tilde{J}''+\frac{k^2}{2}\tilde{\rho}'''
    \end{bmatrix}
\end{eqnarray*}
where 
\begin{equation*}
\alpha^{\pm}=\frac{(J^{\pm})^2}{(P^{\pm})^2}-\gamma (P^\pm)^{\gamma-1}\mbox{ , }\beta^\pm=s-\frac{2 (J^\pm)}{P^\pm}.
\end{equation*}
Finally, wee may rewrite the equation
\begin{equation*}
\lambda \begin{bmatrix}
\tilde{\rho}\\
\tilde{J}
\end{bmatrix} = L_{\pm\infty}\begin{bmatrix}
\tilde{\rho}\\
\tilde{J}
\end{bmatrix}.
\end{equation*}
as a first order system $V'=M^{\pm}V$, with $V=[\tilde{\rho},\tilde{J},u_1,u_2]^T$.

\subsection{Essential spectrum}
As it is well known, the spectrum of $L$ consists of two   parts: the essential spectrum and the point spectrum; we start here by investigating the former.\\
To this end, we recall that the dispersion relation can be found from $det(i \xi Id-M)=0$:
\begin{equation}
\label{disper_rel1}
\lambda^2+\xi(\mu \xi - i (s+\beta^\pm))\lambda+\xi^2\Big{(}-\alpha^\pm+\frac{k^2 \xi^2}{2}-s(\beta^\pm+i\mu \xi)\Big{)}=0.
\end{equation}
To simplify notation, in what follows, we are going to drop the superscript of $\alpha$ and $\beta$.\\
\begin{figure}[H]
\label{figure_spectrum}
\begin{center}
\includegraphics[scale=0.7]{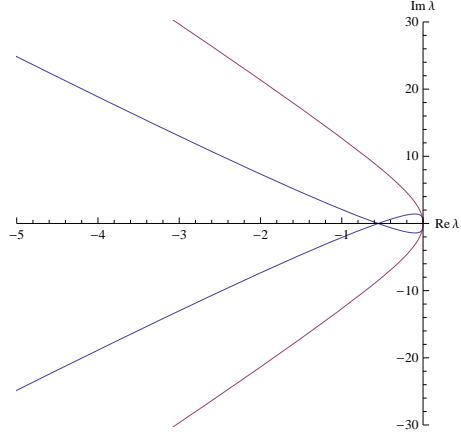}
\end{center}
\caption{The bound for the essential spectrum for parameters $k=0.5,\mu=0.1,s=1,\gamma=3/2,A=1,B=1.3,\rho=0.519,J=-0.418$}
\end{figure}
The essential spectrum is not always stable.
The characteristic equation of $M$ is $det(\nu Id - M)=0$, that is
\begin{equation}
\label{char_eq}
\nu^4+\frac{2 s \mu}{k^2}\nu^3+\frac{2}{k^2}(\alpha+s \beta - \lambda\mu)\nu^2-2\frac{s+\beta}{k^2}\lambda \nu + \frac{2 \lambda^2}{k^2}=0.
\end{equation}

The condition for a subsonic steady-state is $|u|<c_s( P)$, , which becomes
$\alpha<0$ after squaring. The condition $\alpha^{\pm}\leq 0$ corresponds to end states that are either subsonic or sonic.
\begin{lemma}
\label{lemma_essential_spectrum_profile}
If $\mu^2 \neq 2 k^2$ (generically) as long as $\lambda$ is to the right of the curve $\lambda(\xi)$, solving \eqref{disper_rel1}, we have 2 roots with positive real parts and 2 roots with negative real parts of \eqref{char_eq} that is we have consistent splitting. Moreover if $\alpha^\pm \leq 0$, then the bound for the essential spectrum is in the closed left half-plane. Also if $\xi \neq 0$, then $\Re{\lambda_{1,2}}<0$.
\end{lemma}
\begin{proof}
Suppose $\lambda \in \mathbb{R}$, $\lambda \gg 1$. The Discriminant of \eqref{char_eq} is
\begin{equation*}
\Delta = \Big{(} \frac{512}{k^{10}}\mu^4-\frac{2048}{k^8}\mu^2+\frac{2048}{k^6} \Big{)}\lambda^6+\mathcal{O}(\lambda^5).
\end{equation*}
Also 
\begin{equation*}
D = 64 a^4 \Big{(}\hat{s}-\frac{\hat{q}^2}{4}\Big{)}=\Big{(}-\frac{64}{k^4}\mu^2+\frac{128}{k^2}\Big{)}\lambda^2+\mathcal{O}(\lambda),
\end{equation*}
where $a=1$ is the fourth order coefficient of \eqref{char_eq}, $\hat{s}$ and $\hat{q}$ are the zeroth and second order coefficients of the depressed quartic equation, associated to \eqref{char_eq}.\\
If $\mu^2<2 k^2$, then for sufficiently large $\lambda$, $D>0$ and if $\mu^2 \neq 2 k^2$, then for $\lambda \gg 1$, $\Delta>0$, since the leading order terms in the expansion in $\lambda$ have the appropriate signs. Therefore \eqref{char_eq} has two pairs of complex conjugated roots, that are not real. In particular the roots are simple.\\
Since
\begin{equation*}
P=\hat{q}=-\frac{16 \mu}{k^2}\lambda + \mathcal{O}(1),
\end{equation*}
in all cases $P<0$, and if $\mu^2>2 k^2$, then $D<0$, hence the four roots are real and distinct in the regime of real large positive $\lambda$.\\
\\
Now we are going to apply the Descartes' rule of signs in the case of real roots ($\lambda \gg 1$). Suppose first that $s+\beta<0$. Then the number of sign differences between consecutive coefficients is 2, hence there are at most 2 positive roots. If we substitute $\nu \rightarrow -\nu$, then we have again two sign changes, so there are at most 2 negative roots. Suppose that the number of positive roots is less than 2. Then it has to be 0, because it is less than the upper bound by an even number. Then there must be 4 negative roots. This is a contradiction with the upper bound. Hence the polynomial has 2 negative and 2 positive roots.\\
Suppose $s+\beta>0$. Then the number of sign chages is 2. In the case $\nu\rightarrow-\nu$ the number of sign changes
is again 2. Hence similarly as before we get that the polynomial has 2 positive and 2 negative roots.\\
Now, suppose $s+\beta=0$. Then there are 2 sign changes. This is the case also for $\nu\rightarrow -\nu$. Hence again the polynomial has 2 positive and 2 negative roots.\\
\\
Now consider the case $\mu^2<2 k^2$ for real $\lambda \gg 1$. Suppose \eqref{char_eq} does not have a purely imaginary root. The Routh stability criterion is a necessary and sufficient condition for the existence of roots only in the left half-plane. If this is the case, all the coefficients of \eqref{char_eq} must be positive. However the second order coefficient is negative. Hence there are roots in the open right half-plane. If we substitute $\nu \rightarrow -\nu$ the second order coefficient is again negative, therefore \eqref{char_eq} must have roots also in the left half-plane. Hence there are two complex conjugate roots in the left half-plane and two complex conjugated roots in the right half-plane.\\
Suppose we have a polynomial $\nu^4 + a_1 \nu^3 + a_2 \nu^2 + a_3 \nu + a_4$ with all $a_i \neq 0$, $a_i$ - real. If this polynomial has a purely imaginary root, then $\frac{a_3}{a_1}>0$ and $a_3^2+a_1^2 a_4 = a_1 a_2 a_3$.\\
Let $s + \beta>0$. Then all $a_i \neq 0$, $a_1>0$ and $a_3<0$, hence $\frac{a_3}{a_1}<0$, hence \eqref{char_eq} does not have a purely imaginary root.\\
Now, let $s+\beta<0$. In this case $a_i \neq 0$ and $\frac{a_3}{a_1}>0$. However in the regime $\lambda \gg 1$, which we are considering:
\begin{equation*}
a_3^2+a_1^2 a_4 - a_1 a_2 a_3 = \frac{4}{k^6}(k^2(s+\beta)^2+2 s^2\mu^2-2 s(s+\beta)\mu^2)\lambda^2+\mathcal{O}(\lambda).
\end{equation*}
The coefficient of $\lambda^2$ is sum of three positive terms, therefore it is positive. Hence $a_3^2+a_1^2 a_4 \neq a_1 a_2 a_3$. Therefore also in this case \eqref{char_eq} does not have a purely imaginary root.\\
Now, suppose $s+\beta=0$. It follows from the factorization $(\nu^2+p)(\nu^2+q\nu+r)$, with $p \geq 0$ that $\nu^4+a_1 \nu^3 + a_2 \nu^2 + a_4=0$, with $a_i \neq 0$, $a_i$ real, does not have a purely imaginary root. If it would have a purely imaginary root, then this would imply $p=0$ or $q=0$. But we would have $a_4=0$ in the first case and $a_1$=0 in the second. We have $a_1=2 s \mu/k^2 \neq 0$ and $a_4=2 \lambda^2/k^2 \neq 0$, so this is not the case.\\
It follows that \eqref{char_eq} does not have a purely imaginary root.\\
\\
Now we are going to derive a sufficient condition for the stability of the essential spectrum. It applies to constant steady-state and also to the stability of the asymptotic steady-states as $y \rightarrow \pm \infty$. Note that the constant steady-state is not always stable.\\
The roots of the dispersion relation \eqref{disper_rel1} are
\begin{equation*}
\lambda_{1,2}=\frac{1}{2}\Big{(}-\mu \xi^2 + i (s + \beta)\xi \pm \sqrt{D}\Big{)}
\end{equation*}
where
\begin{equation*}
D = \Big{(} \xi (\mu \xi -i (s+\beta))\Big{)}^2-4 \xi^2 \Big{(}-\alpha + \frac{k^2 \xi^2}{2} -s (\beta + i \mu \xi )\Big{)}=p+i q,
\end{equation*}
with
\begin{align*}
p &= -s^2 \xi^2 +4 \alpha \xi^2 + 2 s \beta \xi^2-\beta^2 \xi^2 -2 k^2 \xi^4 + \mu^2 \xi^4,\\
q &= 2 s \mu \xi^3 - 2 \beta \mu \xi^3.
\end{align*}
The condition
\begin{equation}
\label{cond_stab1}
 -\mu \xi^2 + |\Re \sqrt{D}| \leq 0 
 \end{equation}
 guarantees that $\lambda_{1,2}$ are in the closed left half-plane. Since $\mu \xi^2 \geq 0$, it is equivalent to
 \begin{equation}
 \label{cond_stab2}
 (\Re \sqrt{D})^2 \leq \mu^2 \xi^4.
 \end{equation}
 We have
 \begin{equation*}
 \Re \sqrt{D} = \frac{\sqrt{2}}{2} \sqrt{\sqrt{p^2+q^2}+p}
 \end{equation*}
The condition \eqref{cond_stab2} is equivalent to
 \begin{equation}
 \label{cond_stab3}
 \sqrt{p^2+q^2}\leq 2 \mu^2 \xi^4 - p.
 \end{equation}
 Let
 \begin{equation*}
 A = 2 \mu^2 \xi^4 - p = \xi^2 (s-\beta)^2-4 \alpha \xi^2 + 2 k^2 \xi^4 + \mu^2 \xi^4.
 \end{equation*}
 If $\alpha \leq 0$, then $A \geq 0$ for all $\xi \in \mathbb{R}$ as a sum of four non-negative terms. Also, if $\xi \neq 0$, $A>0$. If $A \geq 0$, then \eqref{cond_stab3} is equivalent to $A^2 - p^2 - q^2 \geq 0$. Let $B = A^2 - p^2 - q^2$. Then \eqref{cond_stab3} is equivalent to $B \geq 0$. We have
 \begin{equation*}
 \frac{B}{8 \mu^2}=-2 \alpha \xi^6 + k^2 \xi^8.
 \end{equation*}
 Therefore if $\alpha \leq 0$, $A \geq 0$ and $B \geq 0$. Also if $\xi \neq 0$, $B>0$. Hence \eqref{cond_stab1} holds.
\end{proof}

\subsection{Point spectrum for a profile}
In contrast with the situation described in the proposition above, the localization of the point spectrum of the linearized operator along a profile is more involved (note that our operatori is not self adjoint), and in particular the energy estimate in Lemma \ref{lemma_estimate_real_part_profile} is proved only for $\Re(\lambda)$ sufficiently big.
Thus, to locate the point spectrum in that case, an efficient method is to locate   the zeros of the Evans function,  the latter being exactly the eigenvalues of the operator under consideration. The argument needed requires a careful analysis of the behavior of such function  for large $|\lambda|$, which gives a quantitative version of  the asymptotic results of \cite{Sandstede}, excluding the presence of eigenvalues for $|\lambda|> C$ with an \emph{explicit} bound for the constant (see Lemma \ref{upper_bound_lambda}). To give the definition of the Evans function we shall use later on, we first rewrite  here below our problem in itegrated variables.
\subsubsection{System in integrated variables}
For the analysis of the eigenvalue problem \eqref{eq_variable_coeff} we will need the Evans function and to this end it is also important to re-express the above linearized system in terms of integrated variables, 
because  this transformation removes the zero eigenvalue (always present, being its eigenfunction given by the derivative of the profile), without further modifications of the spectrum; see, for instance \cite{Humpherys}.
To this end, consider
\begin{equation*}
\hat{\rho}(x)=\int_{-\infty}^x \tilde{\rho}(y)dy,\mbox{ }\hat{J}(x)=\int_{-\infty}^x \tilde{J}(y) dy,
\end{equation*} 
Integrating the equation \eqref{eq_variable_coeff} it follows that for $\lambda \neq 0$ the integrated variables $\tilde{\rho}$ and $\tilde{J}$ decay exponentially as $|x|\rightarrow +\infty$. Expressing $\tilde{\rho}$ and $\tilde{J}$ in terms of $\hat{\rho}$ and $\hat{J}$, and integrating \eqref{eq_variable_coeff} from $-\infty$ to $x$ we get the system in integrated variables:
\begin{align}
\label{sys_integrated_variables}
\lambda \hat{\rho} &= s \hat{\rho}'-\hat{J}',\\
\label{sys_integrated_variables1}
\lambda \hat{J} &= f_1 \hat{\rho}' + f_2 \hat{J}'+\mu \hat{J}''+\frac{k^2}{2}\hat{\rho}'''-2 k^2 (\sqrt{P})'\Big{(}\frac{\hat{\rho}'}{\sqrt{P}}\Big{)}',
\end{align}
with
\begin{align*}
f_1(x) &= \frac{J(x)^2}{P(x)^2}-\gamma P(x)^{\gamma-1},\\
f_2(x) &= s - 2 \frac{J(x)}{P(x)}.
\end{align*}
We can rewrite \eqref{sys_integrated_variables}-\eqref{sys_integrated_variables1} as $V'=\hat{M}(x,\lambda)V$, where $V=[\hat{\rho},\hat{J},u_1,u_2]^T$ and 
\begin{equation}
\label{mat_M1}
\hat M(x,\lambda) = \begin{bmatrix}
0 & 0 & 1 & 0\\
-\lambda & 0 & s & 0\\
0 & 0 & 0 & 1\\
\frac{2 \lambda f_2}{k^2} & \frac{2 \lambda}{k^2} & \frac{2 \lambda \mu}{k^2}-\frac{2 f_1}{k^2}-\frac{2 s f_2}{k^2}-\frac{(P')^2}{P^2}& \frac{2 P'}{P}-\frac{2 s \mu}{k^2}
\end{bmatrix}.
\end{equation}
The limit of $\hat{M}(x,\lambda)$ as $x \rightarrow \pm \infty$ is given by
\begin{equation}
\label{mat_M}
M^{\pm} = \begin{bmatrix}
0 & 0 & 1 & 0\\
-\lambda & 0 & s & 0\\
0 & 0 & 0 & 1\\
\frac{2 \beta^\pm \lambda}{k^2} & \frac{2 \lambda}{k^2} & \frac{2}{k^2}(\mu \lambda - s \beta^\pm - \alpha^\pm) & -\frac{2 s \mu}{k^2}
\end{bmatrix},
\end{equation}

\subsubsection{The Evans function}
To define  the Evans function, let us consider the equation $Y'=\hat{M}(y,\lambda)Y$, where $\hat{M}(y,\lambda)$ is defined in \eqref{mat_M1}.
As it is manifest, its limits at $\pm\infty$ are given by  the matrices $M^{\pm}$, defined by \eqref{mat_M}, corresponding to limit states $P^{\pm}$, and we assume  these matrices are hyperbolic. This is always true if we are to the right of the bound for the essential spectrum. In addition, we assume that $M^{-}$ has $k$ unstable eigenvalues $\nu^{-}_1,\dots,\nu^{-}_k$ (i.e.\ $\Re(\nu^{-}_i)>0$), and $M^{+}$ has $n-k$ stable eigenvalues $\nu^{+}_1,...,\nu^{+}_{n-k}$ (i.e.\ $\Re(\nu^{+}_i)<0$), and denote the corresponding (normalized) eigenvectors by $v^{\pm}_i$. In our case $n=4$ and $k=2$. Let $Y^{-}_i$ be a solution of $Y'=M(y,\lambda)Y$, satisfying $exp(\nu^- y)Y^{-}(y)$ tends to $v^{-}_i$ as $y \rightarrow -\infty$ and $exp(\nu^+ y)Y^{+}(y)$ tends to $v^{+}_i$ as $y \rightarrow +\infty$. Then,  the Evans function can be defined by
\begin{equation*}
E(\lambda) = det(Y^-_1(0), .., Y^-_k(0),  Y^+_1(0), ... ,Y^+_{n-k}(0)).
\end{equation*}
As a consequence, a point $\lambda \in \mathbb{C}$ is in the point spectrum of $L$ if and only if $E(\lambda)=0$.
\subsubsection{Estimate for $\Re(\lambda)$ for variable coefficients}
Now we consider the eigenvalue equation \eqref{eq_variable_coeff} with variable coefficients. 
We are going to derive a bound of the real part of the eigenvalues $\Re(\lambda)$ using an energy estimate. 
\begin{lemma}
\label{lemma_estimate_real_part_profile}
For each eigenvalue of \eqref{eq_variable_coeff} with variable profile we have\\ $\Re (\lambda) \leq \max(C_1,C_2, 2 \frac{C_3}{k^2})$, where
\begin{eqnarray*}
C_1 = \frac{1}{2 \epsilon_1}+\frac{M_2}{\epsilon_2}+\frac{M_4}{\epsilon_3},\mbox{ }
C_2 = \frac{\epsilon_2}{2}+M_3,\mbox{ }
C_3 = \frac{M_1}{\epsilon_2}+\frac{M_5}{\epsilon_3},
\end{eqnarray*}
with $\epsilon_2>0$, $\epsilon_1/2+\epsilon_3/2<\mu$, and explicit constants $M_i$.
\end{lemma}
\begin{proof}
We multiplying the first equation of \eqref{eq_variable_coeff} by $\bar{\tilde{\rho}}$ and the second equation of \eqref{eq_variable_coeff} by $\bar{\tilde{J}}$.
Let $\tilde{\rho}=\tilde{\rho_r}+i \tilde{\rho_i}$ and $\tilde{J}=\tilde{J_r}+i \tilde{J_i}$. 
We have
\begin{align}
\label{eq_deriv_1}
-\alpha s \Re(\tilde{\rho}'\overline{\tilde{\rho}})&=-\alpha s(\tilde{\rho_r}' \tilde{\rho_r}+\tilde{\rho_i}' \tilde{\rho_i})=-\frac{\alpha s}{2}(|\tilde{\rho}|^2)',\\
\label{eq_deriv_2}
\beta \Re(\tilde{J}'\overline{\tilde{J}}) &= \frac{\beta}{2}(|\tilde{J}|^2)'
\end{align}
By integration by parts we obtain 
\begin{equation}
\label{eq_const4}
\int \mu(\tilde{J_r}''\tilde{J_r}+\tilde{J_i}''\tilde{J_i})dy = - \mu \int ((\tilde{J_r}')^2+(\tilde{J_i}')^2)dy=-\mu \int |\tilde{J}'|^2 dy.
\end{equation}
We have 
\begin{equation*}
\Re(\tilde{\rho}'''\overline{J})=\tilde{\rho_r}'''\tilde{J_r}+\tilde{\rho_i}'''\tilde{J_i},
\end{equation*}
hence by integration by parts
\begin{equation*}
\int \frac{k^2}{2}\Re(\tilde{\rho}'''\overline{\tilde{J}}) dy=-\frac{k^2}{2}\int (\tilde{J_r}' \tilde{\rho_r}''+\tilde{J_i}' \tilde{\rho_i}'')dy
\end{equation*}
Substituting $\tilde{J_r}'$ and $\tilde{J_i}'$ from the first equation of \eqref{eq_variable_coeff} yields
\begin{align*}
&-\frac{k^2}{2}\int (\tilde{J_r}' \tilde{\rho_r}''+\tilde{J_i}' \tilde{\rho_i}'')dy \nonumber\\
&=-\frac{k^2}{2}\int \Big{(} (s \tilde{\rho_r}'-\Re (\lambda) \tilde{\rho_r}\nonumber\\
&+\Im(\lambda) \tilde{\rho_i})\tilde{\rho_r}''+(s \tilde{\rho_i}'-\Re (\lambda) \tilde{\rho_i}-\Im(\lambda) \tilde{\rho_r})\tilde{\rho_i}''\Big{)}dy\nonumber\\
&=-\frac{k^2}{2}\int (t_1+t_2+t_3)dy,
\end{align*}
where
\begin{align*}
t_1 &= s(\tilde{\rho_r}'' \tilde{\rho_r}'+\tilde{\rho_i}'' \tilde{\rho_i}')=\frac{s}{2}\Big{(}(\tilde{\rho_r}')^2+(\tilde{\rho_i}')^2\Big{)}',\\
t_2 &= \Im(\lambda) (\tilde{\rho_i} \tilde{\rho_r}''-\tilde{\rho_i}'' \tilde{\rho_r}),\\
t_3 &= -\Re(\lambda) (\tilde{\rho_r}\tilde{\rho_r}''+\tilde{\rho_i}\tilde{\rho_i}'').
\end{align*}
Therefore
\begin{equation*}
-\frac{k^2}{2}\int t_1 dy = 0.
\end{equation*}
Using integration by parts we get
\begin{equation*}
-\frac{k^2}{2}\int t_2 dy = -\frac{k^2 \Im(\lambda)}{2}\int (-\tilde{\rho_r}'\tilde{\rho_i}'+\tilde{\rho_i}'\tilde{\rho_r}')dy=0,
\end{equation*}
and again by integration by parts
\begin{equation*}
-\frac{k^2}{2}\int t_3 dy = -\frac{\Re(\lambda) k^2}{2}\int \Big{(} (\tilde{\rho_r}')^2+(\tilde{\rho_i}')^2\Big{)} dy=-\frac{\Re(\lambda) k^2}{2}\int |\tilde{\rho}'|^2 dy,
\end{equation*}
which yields 
\begin{equation}
\label{eq_const5}
\int \frac{k^2}{2}\Re(\tilde{\rho}'''\overline{\tilde{J}}) dy = -\frac{\Re(\lambda) k^2}{2}\int |\tilde{\rho}'|^2 dy.
\end{equation}
Taking the real part of $s  \tilde{\rho}' \overline{\tilde{\rho}}$ and integrating we get
\begin{equation*}
\int \Re(s \tilde{\rho}' \overline{\tilde{\rho}})dy=0
\end{equation*}
by \eqref{eq_deriv_1}. Moreover
\begin{equation*}
\int \Re(-\tilde{J}' \overline{\tilde{\rho}})dy \leq \frac{1}{2 \epsilon_1}\int |\tilde{\rho}|^2 dy + \frac{\epsilon_1}{2}\int |\tilde{J'}|^2 dy
\end{equation*}
by Young inequality. From the second equation of \eqref{eq_variable_coeff} we have
\begin{equation*}
\int \Re(s \tilde{J}' \overline{\tilde{J}})dy=0
\end{equation*}
by \eqref{eq_deriv_2}. Also
\begin{align*}
&\int \Re\Big{(}\Big{(}(\frac{J^2}{P}-\gamma P^{\gamma-1}) \tilde{\rho} \Big{)}' \overline{\tilde{J}} \Big{)} dy \nonumber\\
&\leq \int  \Big{(} \frac{\epsilon_2}{2} |\tilde{J}|^2 +\frac{1}{2 \epsilon_2}(|(f_1 \tilde{\rho_r})'|^2 + |(f_1 \tilde{\rho_i})'|^2) \Big{)}dy \nonumber\\
& \leq  \int  \Big{(} \frac{\epsilon_2}{2} |\tilde{J}|^2 +\frac{1}{\epsilon_2}(|f_2|^2 |\tilde{\rho}|^2+(|f_1|^2 |\tilde{\rho'}|^2)\Big{)}dy \nonumber\\
&\leq  \frac{\epsilon_2}{2} \int |\tilde{J}|^2 dy + \frac{M_2}{\epsilon_2}\int |\tilde{\rho}|^2 dy+\frac{M_1}{\epsilon_2}\int |\tilde{\rho}'|^2 dy,
\end{align*}
where $f_1 = \frac{J^2}{P}-\gamma P^{\gamma-1}$, $f_2 = f_1'$, $M_{1,2} = \sup_{y \in \mathbb{R}}f_{1,2} (y)^2$. Further
\begin{align*}
\int \Re\Big{(}\Big{(}-\frac{2 J}{P} \tilde{J}\Big{)}' \overline{\tilde{J}}\Big{)}dy&=-\int\Big{(}\frac{2 J}{P}\Big{)}\Big{(}\tilde{J_r} \tilde{J_r}'+\tilde{J_i} \tilde{J_i}'\Big{)}dy \nonumber\\
=\int \Big{(}\frac{J}{P}\Big{)} (|\tilde{J}|^2)'dy &=-\int \Big{(}\frac{J}{P}\Big{)}'|\tilde{J}|^2dy \leq M_3 \int |\tilde{J}|^2 dy,
\end{align*}
by integration by parts, where $f_3 = (\frac{J}{P})'$ and $M_3 = \sup_{y \in \mathbb{R}}|f_3 (y) |$. Moreover
\begin{align*}
&\int \Re \Big{(}-2k^2 \Big{(}(\sqrt{P})'\Big{(} \frac{\tilde{\rho}}{\sqrt{P}}\Big{)}' \Big{)}' \overline{\tilde{J}}\Big{)}dy\nonumber\\
& =- \int 2k^2 (\sqrt{P})' \Big{(}\Big{(} \frac{\tilde{\rho_r}}{\sqrt{P}}\Big{)}' \tilde{J_r}' +\Big{(} \frac{\tilde{\rho_i}}{\sqrt{P}}\Big{)}' \tilde{J_i}' \Big{)} dy \nonumber\\
& \leq \int  \Big{(}|2 k^2 (\sqrt{P})' \Big{(} \frac{\tilde{\rho_r}}{\sqrt{P}}\Big{)}' \tilde{J_r}'| +|2 k^2 (\sqrt{P})' \Big{(} \frac{\tilde{\rho_i}}{\sqrt{P}}\Big{)}' \tilde{J_i}'| \Big{)} dy \nonumber\\
&=\int \Big{(}|f_4 \tilde{\rho_r} + f_5 \tilde{\rho_r}'||J_r'|+|f_4 \tilde{\rho_i} + f_5 \tilde{\rho_i}'||\tilde{J_i}'| \Big{)} \nonumber\\
& \leq \frac{\epsilon_3}{2} \int |\tilde{J}'|^2 dy + \frac{1}{2 \epsilon_3}\int \Big{(}|f_4 \tilde{\rho_r} + f_5 \tilde{\rho_r}'|^2+|f_4 \tilde{\rho_i} + f_5 \tilde{\rho_i}'|^2 \Big{)}dy  \nonumber\\
& \leq \frac{\epsilon_3}{2} \int |\tilde{J}'|^2 dy + \frac{1}{\epsilon_3}\int f_4^2 |\tilde{\rho}|^2 dy + \frac{1}{\epsilon_3}\int f_5^2 |\tilde{\rho}'|^2 dy \nonumber \\ 
& \leq \frac{\epsilon_3}{2} \int |\tilde{J}'|^2 dy + \frac{M_4}{\epsilon_3}\int |\tilde{\rho}|^2 dy+\frac{M_5}{\epsilon_3}\int |\tilde{\rho}'|^2 dy,
\end{align*}
where $f_4 = k^2 \frac{P' (\sqrt{P})'}{P^{3/2}}$, $f_5 = 2 k^2 \frac{(\sqrt{P})'}{\sqrt{P}}$, $M_{4,5} = \sup_{y \in \mathbb{R}}f_{4,5} (y)^2$. Using \eqref{eq_const4} and \eqref{eq_const5} and collecting the inequalities we get
\begin{align}
&(\Re(\lambda)-C_1)\int |\tilde{\rho}|^2 dy + (\Re(\lambda)-C_2)\int |\tilde{J}|^2 dy \nonumber\\
\label{estimate_re_lambda}
&+ (\frac{k^2}{2}\Re(\lambda)-C_3)\int |\tilde{\rho}'|^2 dy \leq (-\mu + \frac{\epsilon_1}{2}+\frac{\epsilon_3}{2})\int |\tilde{J}'|^2 dy
\end{align} 
Hence if $\frac{\epsilon_1}{2}+\frac{\epsilon_3}{2} < \mu$ and $\Re(\lambda)>\max(C_1,C_2, 2 \frac{C_3}{k^2})$ the left hand side of \eqref{estimate_re_lambda} is positive and the right hand side is negative, which is a contradiction. Therefore $\Re (\lambda) \leq \max(C_1,C_2, 2 \frac{C_3}{k^2})$.
\end{proof}
\subsubsection{The Evans function for large $|\lambda|$}
Now we are going to show that we can consider a scalar equation to locate the eigenvalues.\\
Let us consider the eigenvalue problem with linear dispersion:
\begin{align}
\lambda \tilde{\rho} &= s \tilde{\rho}'-\tilde{J}' \label{eigenvalue_linear_dispersion_1},\\
\lambda \tilde{J} &= (f_1(x) \tilde{\rho})' + (f_2(x) \tilde{J})'+\mu \tilde{J}'' + \frac{k^2}{2} \tilde{\rho}'''
\label{eigenvalue_linear_dispersion_2}
\end{align}
with
\begin{align}
f_1(x) &= \frac{J(x)^2}{P(x)^2}-\gamma P(x)^{\gamma-1}, \nonumber\\
f_2(x) &= s - 2 \frac{J(x)}{P(x)} \label{funs}
\end{align}
Consider also
\begin{align}
\hat{\rho}^{(4)}+\frac{2 s \mu}{k^2} \hat{\rho}^{'''} + \frac{2}{k^2}(f_1 + f_2 s - \lambda \mu)\hat{\rho}^{''}+\frac{2}{k^2}(f_1'+s f_2'-\lambda f_2 - s \lambda)\hat{\rho}' \nonumber\\
\label{scalar_ODE}
+\frac{2}{k^2}(-\lambda f_2'+\lambda^2)\hat{\rho}=0 
\end{align}
and the eigenvalue problem with nonlinear dispersion:
\begin{align}
\label{eigenvalue_equation_1}
\lambda \tilde{\rho} &= s \tilde{\rho}'-\tilde{J}',\\
\lambda \tilde{J} &= (f_1(x) \tilde{\rho})' + (f_2(x) \tilde{J})'+\mu \tilde{J}'' + f_3(x) \tilde{\rho} + f_4(x) \tilde{\rho}' \nonumber \\ 
&+ f_5(x) \tilde{\rho}''+ \frac{k^2}{2} \tilde{\rho}''',
\label{eigenvalue_equation_2}
\end{align}
which is equivalent to \eqref{eq_variable_coeff},
where $f_1(x)$ and $f_2(x)$ are given by \eqref{funs}, and
\begin{align*}
f_3(x) &= k^2 \Big{(} \frac{P'(x)P''(x)}{P(x)^2} - \Big{(}\frac{P'(x)}{P(x)}\Big{)}^3 \Big{)}, \nonumber\\
f_4(x) &= k^2 \Big{(} \frac{3}{2}\Big{(} \frac{P'(x)}{P(x)}\Big{)}^2-\frac{P''(x)}{P(x)}\Big{)}, \nonumber\\
f_5(x) &= -k^2 \frac{P'(x)}{P(x)},
\end{align*}
and the scalar equation
\begin{align}
&\hat{\rho}^{(4)} + \frac{2}{k^2}(s \mu + f_5)\hat{\rho}^{'''}+\frac{2}{k^2}(f_1 + f_2 s + f_4 -\lambda \mu) \hat{\rho}^{''}\nonumber \\ 
&+\frac{2}{k^2}(f_1'+f_2' s-\lambda f_2+f_3-s \lambda) \hat{\rho}^{'}
+\frac{2}{k^2}(\lambda f_2'+\lambda^2)\hat{\rho}=0. \label{scalar_ODE_nonlinear_dispersion}
\end{align}
In the proof of the following lemma we rescale the variable $y$.
\begin{lemma}
If $\lambda \neq 0$, the eigenvalue equations \eqref{eigenvalue_linear_dispersion_1}-\eqref{eigenvalue_linear_dispersion_2},  \eqref{eigenvalue_equation_1}-\eqref{eigenvalue_equation_2} are  equivalent to \eqref{scalar_ODE} and  \eqref{scalar_ODE_nonlinear_dispersion}, respectively. In particular, if $\lambda \neq 0$ is not an eigenvalue of \eqref{scalar_ODE}(\eqref{scalar_ODE_nonlinear_dispersion}), it is also not an eigenvalue of \eqref{eigenvalue_linear_dispersion_1}-\eqref{eigenvalue_linear_dispersion_2}(\eqref{eigenvalue_equation_1}-\eqref{eigenvalue_equation_2}). The Evans function for \eqref{scalar_ODE}(\eqref{scalar_ODE_nonlinear_dispersion})  does not vanish for $\Re(\lambda)\geq 0$ and $|\lambda|$ large enough.
\end{lemma}
\begin{proof}
Suppose that $\lambda \neq 0$. Integrating \eqref{eigenvalue_linear_dispersion_1}-\eqref{eigenvalue_linear_dispersion_2} we get
\begin{equation*}
\int \tilde{\rho} dx=0,\mbox{ }\int \tilde{J} dx = 0.
\end{equation*}
We shall use the integrated variable
\begin{equation*}
\hat{\rho}(x)=\int_{-\infty}^x \tilde{\rho}(y)dy.
\end{equation*}
We have that $\tilde{\rho}(x)$ decays exponentially as $|x|\rightarrow +\infty$. In particular \\$|\tilde{\rho}(x)
| \leq C_1 \exp(-C_2 x)$. So
\begin{align*}
|\hat{\rho}(x)|&=|\int_{x}^\infty \tilde{\rho}(y)dy| \leq \int_{x}^\infty |\tilde{\rho}(y)|dy \leq C_1 \int_{x}^\infty \exp(-C_2 y) dy \\ &=\frac{C_1}{C_2}\exp(-C_2 x).
\end{align*}
Therefore $\hat{\rho}(x)$ decays exponentially as $x \rightarrow +\infty$. Similarly $\hat{\rho}(x)$ decays exponentially as $x \rightarrow -\infty$.\\
\\
Integrating equation \eqref{eigenvalue_linear_dispersion_1} from $-\infty$ to $x$ yields
\begin{equation}
\label{expr_J}
\tilde{J} = s \hat{\rho}'-\lambda \hat{\rho}.
\end{equation}
After expressing $\tilde{\rho}$ in terms of $\hat{\rho}$, substituting $\tilde{J}$ from \eqref{expr_J} and dividing by $k^2/2$ we get \eqref{scalar_ODE}. If $\lambda \neq 0$  is not an eigenvalue of \eqref{scalar_ODE}, it is also not an eigenvalue of \eqref{eigenvalue_linear_dispersion_1}-\eqref{eigenvalue_linear_dispersion_2}.\\
\\
Now, we make a change of variable
\begin{equation}
\label{chnage_var}
x = \frac{y}{|\lambda|^\frac{1}{2}}.
\end{equation}
Dividing \eqref{scalar_ODE} by $|\lambda|^2$ yields
\begin{align}
&\frac{d^4 \hat{\rho}}{dy^4}+\frac{2 s \mu}{k^2 |\lambda|^\frac{1}{2}}\frac{d^3 \hat{\rho}}{dy^3}
+ \frac{2}{k^2}\Bigg{(}\frac{f_1\Big{(}\frac{y}{|\lambda|^\frac{1}{2}}\Big{)} + f_2\Big{(}\frac{y}{|\lambda|^\frac{1}{2}}\Big{)} s}{|\lambda|} - \frac{\lambda}{|\lambda|} \mu\Bigg{)}\frac{d^2 \hat{\rho}}{dy^2} \nonumber\\
&+\frac{2}{k^2 |\lambda|^\frac{3}{2}}\Big{(}f_1'\Big{(}\frac{y}{|\lambda|^\frac{1}{2}}\Big{)}+s f_2'\Big{(}\frac{y}{|\lambda|^\frac{1}{2}}\Big{)}-\lambda f_2\Big{(}\frac{y}{|\lambda|^\frac{1}{2}}\Big{)} - s \lambda\Big{)}\frac{d \hat{\rho}}{dy} \nonumber\\
&+\frac{2}{k^2}\Bigg{(}-\frac{\lambda}{|\lambda|^2} f_2'\Big{(}\frac{y}{|\lambda|^\frac{1}{2}}\Big{)}+\frac{\lambda^2}{|\lambda|^2}\Bigg{)}\hat{\rho}=0
\label{rescaled_system}
\end{align}
Denote $\tilde{\lambda} = \frac{\lambda}{|\lambda|}$. Taking $\lim_{|\lambda|\rightarrow +\infty}$ in \eqref{rescaled_system} yields:
\begin{equation}
\label{constant_coeff_equation}
\frac{d^4 \hat{\rho}}{dy^4} - \frac{2 \mu}{k^2}\tilde{\lambda} \frac{d^2 \hat{\rho}}{dy^2} +\frac{2}{k^2}\tilde{\lambda}^2 \hat{\rho}=0.
\end{equation}
Rewriting \eqref{constant_coeff_equation} as a first-order system gives:
\begin{equation}
\label{first_order_constant_coefficient}
\frac{d}{dy}\begin{bmatrix}
\rho\\u_1\\u_2\\u_3
\end{bmatrix}
 = \begin{bmatrix}
0 & 1 & 0 & 0\\
0 & 0 & 1 & 0\\
0 & 0 & 0 & 1\\
-\frac{2}{k^2}\tilde{\lambda} & 0 & \frac{2 \mu}{k^2} \tilde{\lambda} & 0
\end{bmatrix}
\begin{bmatrix}
\rho\\u_1\\u_2\\u_3
\end{bmatrix}.
\end{equation}
The characteristic equation of \eqref{first_order_constant_coefficient} is
\begin{equation}
\label{char_eq_1}
z^4 - \frac{2 \mu}{k^2}\tilde{\lambda} z^2 + \frac{2 \tilde{\lambda}^2}{k^2}=0.
\end{equation}
Let $\Re(\lambda) \geq 0$. If 
\begin{equation*}
D = \frac{4}{k^2}\big{(} \frac{\mu^2}{k^2}-2)\tilde{\lambda}^2 \neq 0,
\end{equation*}
that is $\frac{\mu^2}{k^2} \neq 2$, then \eqref{char_eq_1} has 4 distinct roots, with $\Re(z_1),\Re(z_2)<0$ and \\$\Re(z_3),\Re(z_4)>0$.\\
The condition $D \neq 0$ holds, when the dispersion and dissipation  terms do not exactly balance.\\
We make the change of variable $w=z^2$. Then \eqref{char_eq_1} becomes:
\begin{equation}
\label{quadratic_char_eq}
w^2 - \frac{2 \mu}{k^2}\tilde{\lambda} w + \frac{2 \tilde{\lambda}^2}{k^2}=0
\end{equation}
If $\Re(\lambda) \geq 0$, then $\Re(\tilde{\lambda}) \geq 0$. If $D \neq 0$, since $\frac{2 \tilde{\lambda}^2}{k^2} \neq 0$, the equation \eqref{quadratic_char_eq} has two distinct nonzero roots. Suppose $\frac{\mu^2}{k^2} < 2$, that is the dispersion is dominating. The roots are:
\begin{equation*}
w_{1,2} = \frac{\mu \tilde{\lambda}}{k^2} \pm \frac{i}{k}\sqrt{2-\frac{\mu^2}{k^2}}\tilde{\lambda}.
\end{equation*}
Moreover $|w_j| = \sqrt{2}/k$, $j=1,2$. Let $\lambda = \exp(i \theta)$, and $\theta \in [0, \pi/2] \cup [3 \pi/2, 2 \pi[$. We have
\begin{equation*}
Arg\Big{(}\frac{\mu}{k^2} + \frac{i}{k}\sqrt{2-\frac{\mu^2}{k^2}}\Big{)} \in ]0, \pi/2[.
\end{equation*}
If $\theta \in [0, \pi/2]$, then $Arg(w_1) \in ]0, \pi[$. If $\theta \in [3 \pi/2, 2 \pi[$, then $\Re(w_1) > 0$. Also
\begin{equation*}
Arg\Big{(}\frac{\mu}{k^2} - \frac{i}{k}\sqrt{2-\frac{\mu^2}{k^2}}\Big{)} \in ]3 \pi/2, 2 \pi[.
\end{equation*} 
If $\theta \in [0, \pi/2]$, then $\Re(w_2) > 0$. If $\theta \in [3 \pi/2, 2 \pi[$, then $Arg(w_1) \in ]\pi, 2 \pi[$. So $w_{1,2}$ are not negative. Hence $z^2 = w_k$ has one solution with positive and one with negative real part for $k=1,2$. Therefore $z_j$ are not purely imaginary. Also $z_j$ are distinct, because distinct nonzero numbers cannot have equal square roots.\\
The equation \eqref{constant_coeff_equation} has consant coefficients, so its Evans function can be computed.
Let $z_k$  be a simple eigenvalue of the matrix
\begin{equation}
\label{matrix_first_order_constant_coefficient}
 \begin{bmatrix}
0 & 1 & 0 & 0\\
0 & 0 & 1 & 0\\
0 & 0 & 0 & 1\\
-\frac{2}{k^2}\tilde{\lambda} & 0 & \frac{2 \mu}{k^2} \tilde{\lambda} & 0
\end{bmatrix}.
\end{equation}
The associated eigenvector is $v_k = [1,z_k,z_k^2,z_k^3]^T$. The Evans function for \eqref{first_order_constant_coefficient} is 
\begin{align*}
&\tilde{E}(\lambda) = \det([v_1, v_2,v_3,v_4]) = \prod_{j<k}(z_j-z_k) \nonumber\\
&=(z_1-z_2)(z_1-z_3)(z_1-z_4)(z_2-z_3)(z_2-z_4)(z_3-z_4) \neq 0. 
\end{align*}
We have that $\tilde{E}(\lambda) \neq 0$, since the eigevalues $z_k$ are distinct. The coefficients of \eqref{constant_coeff_equation} and \eqref{rescaled_system} are uniformly in $y$ close to each other, their Evans functions are uniformly close in $\lambda$ as a consequence of Theorem 3.1, \cite{Sandstede}. Therefore the Evans function for \eqref{scalar_ODE} never vanishes for $\Re(\lambda)\geq 0$, and $|\lambda|>C$, where $C$ is some constant. So for any eigenvalue with $\Re(\lambda)\geq 0$, we have $|\lambda| \leq C$.\\
\\
If the viscosity $\mu=0$, choosing $\tilde{\lambda}=i$ yields purely imaginary roots of \eqref{char_eq_1}. Therefore the matrix \eqref{matrix_first_order_constant_coefficient} is not hyperbolic.\\
We rewrite \eqref{eq_variable_coeff} as \eqref{eigenvalue_equation_1}
Expressing $\tilde{\rho}$ in terms of $\hat{\rho}$, which decays exponentially as $|x| \rightarrow + \infty$, integrating the first equation of \eqref{eigenvalue_equation_1} and substituting $\tilde{J}$ in the second equation of \eqref{eigenvalue_equation_1} we get \eqref{scalar_ODE_nonlinear_dispersion}.
Making the change of variable \eqref{chnage_var} in \eqref{scalar_ODE_nonlinear_dispersion} yields
\begin{align}
&\frac{d^4 \hat{\rho}}{dy^4}+\frac{2}{k^2}\frac{s \mu + f_5}{|\lambda|^\frac{1}{2}}\frac{d^3 \hat{\rho}}{dy^3}+\frac{2}{k^2}\Big{(}\frac{f_1 + f_2 s + f_4}{|\lambda|} -\mu \frac{\lambda}{|\lambda|}\Big{)}\frac{d^2 \hat{\rho}}{dy^2} \nonumber\\
&+\frac{2}{k^2}\frac{f_1'+f_2' s-\lambda f_2+f_3-s \lambda}{|\lambda|^\frac{3}{2}}\frac{d \hat{\rho}}{dy}
+\frac{2}{k^2}\Big{(} \frac{\lambda f_2'}{|\lambda|^2} +\frac{\lambda^2}{|\lambda|^2}\Big{)}\hat{\rho}=0,\label{rescaled_equation_nonlinear_dispresion}
\end{align}
where all the functions $f_k$ and their derivatives are evaluated at $y/|\lambda|^\frac{1}{2}$. Taking limit as $|\lambda| \rightarrow +\infty$ in \eqref{rescaled_equation_nonlinear_dispresion} we obtain \eqref{constant_coeff_equation}. Therefore the same conclusion follows. \end{proof}
\subsubsection{Estimate for the maximum of $|\lambda|$}
We will estimate the constant $C$, such that each eigenvalue of \eqref{eq_variable_coeff} with non-negative real part has absolute value less than $C$. In preparation to stating Lemma \ref{upper_bound_lambda}, let
\begin{align*}
\theta_{1,2} &= arg\Big(\frac{\mu}{2}\pm i \Big{(} 1 - \frac{\mu^2}{4}\Big{)}^\frac{1}{2}\Big), \nonumber\\
\alpha &= \Re(exp(i(\theta+\theta_1)/2))\mbox{ for }\theta=\pi/2, \nonumber\\
z_{1,3} &= \mp \exp(i(\theta+\theta_1)/2),\mbox{ }z_{2,4} = \mp \exp(i(\theta+\theta_2)/2).
\end{align*}
The distance between the roots $|z_j-z_k|$ does not depend on $\theta$ and we can compute it e.g. for $\theta = 0$.\\
Also
\begin{align}
m_1(\lambda) &= |\lambda|^{-1}\sup_{x\geq 0}|f_2'|, \nonumber\\
m_2(\lambda) &= |\lambda|^{-\frac{3}{2}}\sup_{x\geq 0}|f_1'+s f_2'+f_3| + |\lambda|^{-\frac{1}{2}}\sup_{x\geq 0}|f_2+s|, \nonumber\\
m_3(\lambda) &= |\lambda|^{-1} \sup_{x\geq 0}|f_1+s f_2 + f_4|, \nonumber\\
m_4(\lambda) &= |\lambda|^{-\frac{1}{2}}\sup_{x \geq 0}|s \mu + f_5|.
\label{bounds}
\end{align}
Moreover
\begin{align}
r_{j,k} &= \frac{p_k}{g_j}, p_k = \sum_{l=1}^4 m_l(\lambda)|z_l|^{l-1}=\sum_{l=1}^4 m_l(\lambda), \nonumber\\
g_1 &= |z_1-z_2||z_1-z_3||z_1-z_4|, g_2 = |z_1-z_2||z_2-z_3||z_2-z_4|, \nonumber\\
g_3 &= |z_1-z_3||z_2-z_3||z_3-z_4|, g_4 = |z_1-z_4||z_2-z_4||z_3-z_4|, \nonumber\\
R_{+} &= [r_{j,k}].
\label{residual_matrix}
\end{align}
The matrix $R_{-}$ has the suprema in the definition of $m_k(\lambda)$ taken for $x \leq 0$.
\begin{eqnarray*}
\delta_{\pm} = \Vert R_{\pm} \Vert_F,\mbox{ }
\epsilon_{\pm} = 4 \alpha^{-1} \delta_{\pm}.
\end{eqnarray*}
In the following lemma we decompose the system into a constant coefficients part, which depends only on the direction $\lambda/|\lambda|$, and a perturbation, which becomes small for large $|\lambda|$. We use exponential dichotomies to estimate the difference between the Evans functions of the constant coefficient system and the perturbed system. We suppose for simplicity that $k=\sqrt{2}$.
\begin{lemma}\label{upper_bound_lambda}
Suppose $\mu < 2$. Let $\sqrt{2}\sqrt{\epsilon_{+}^2+\epsilon_{-}^2}<1$, and $\delta_{\pm}<\frac{\alpha}{4}$ for some $|\lambda|=C$. Then for all $\lambda$ with $\Re(\lambda)\geq 0$ and $|\lambda|\geq C$, the Evans function $E(\lambda)$ for \eqref{scalar_ODE_nonlinear_dispersion} has no zeroes.
\end{lemma}
\begin{proof}
Consider the system
\begin{equation}
\label{system_constant_coefficients1}
\frac{du}{dx}=A(\tilde{\lambda})u
\end{equation}
where $\tilde{\lambda} = \frac{\lambda}{|\lambda|}$. The matrix $A(\tilde{\lambda})$ does not depend on $x$ and has simple eigenvalues $z_j$, $j=1,...,4$, with $\Re(z_1), \Re(z_2) < 0$ and $\Re(z_3), \Re(z_4) > 0$. We may fix $\tilde{\lambda}$. The system \eqref{system_constant_coefficients1} has an exponential dichotomy (see \cite{Coppel}, Chapter 4) on $R_{0}^+$. That is there are constants $K$, $\alpha$ and projection $P$ such that
\begin{align*}
\|X(t)PX^{-1}(s)\|_2 \leq K e^{-\alpha (t-s)}, t \geq s &\geq 0,\\
\|X(t)(Id-P)X^{-1}(s)\|_2 \leq K e^{-\alpha (s-t)}, s \geq t &\geq 0,
\end{align*}
where $X(t)$ is the fundamental solutions matrix for \eqref{system_constant_coefficients1} with $X(0)=Id$. We consider the scalar product $\langle x,y \rangle= x \cdot \bar{y}$. The vector norm is $|x| = \sqrt{\langle x, x \rangle}$. The 2-norm is $\|A\|_2 = \sup_{|x|=1}|A x|$. Since the matrix $A(\tilde{\lambda})$ has constant coefficients, the constants $K$ and $\alpha$ can be computed.\\
Now, consider the perturbed system
\begin{equation}
\label{perturbed_system}
\frac{du}{dx}=A(\tilde{\lambda})u + B(x,\lambda) u.
\end{equation}
We have
\begin{equation*}
\lim_{| \lambda | \rightarrow +\infty}\sup_{x \geq 0} \|B(x,\lambda)\|_2 = 0.
\end{equation*}
More precisely, $\sup_{x \geq 0} \|B(x,\lambda)\|_2 = \mathcal{O}(|\lambda|^{-\frac{1}{2}})$. Let $\delta = \sup_{x \geq 0} \|B(x,\lambda)\|$. If $\delta < \alpha/(4 K^2)$, then the perturbed system \eqref{perturbed_system} also has an exponential dichotomy with projection $Q$. Moreover (see \cite{Coppel}, Chapter 4, Prop.\ 1)
\begin{equation*}
\|P-Q\|_2 \leq 4 \alpha^{-1} K^3 \delta.
\end{equation*}
Let $P$ and $Q$ be projections onto the subspaces $M$ and $N$. There exist unique orthogonal projections $\tilde{P}$ and $\tilde{Q}$ onto $M$ and $N$. Moreover (see \cite{Kato} p.58, Theorem 6.35)
\begin{equation*}
\| \tilde{P}-\tilde{Q} \|_2 \leq \| P - Q \|_2.
\end{equation*}
We have $\dim M = \dim N = 2$. Denote $\epsilon = 4 \alpha^{-1} K^3 \delta$. Then $\| \tilde{P}-\tilde{Q} \|_2 \leq \epsilon$. Let $u_k$, $k=1,...,4$ be the eigenvectors of $A(\tilde{\lambda})$. We suppose they are normalized, that is $|u_k| = 1$. We have $u_k \in M$, that is $\tilde{P} u_k = u_k$. Denote $h_k = u_k - \tilde{Q} u_k$. Then
\begin{equation*}
|h_k| =|u_k - \tilde{Q} u_k | = | \tilde{P} u_k - \tilde{Q} u_k | \leq \| \tilde{P} - \tilde{Q}\|_2 |u_k | \leq \epsilon.
\end{equation*}
Let $v_k = \tilde{Q} u_k$. We have $\langle v_k, h_k \rangle = \langle \tilde{Q} u_k, u_k - \tilde{Q} u_k \rangle = 0$. Hence $|u_k|^2 = |v_k|^2 + |h_k|^2$. Therefore $|v_k|^2 \geq |u_k|^2 - \epsilon^2 = 1-\epsilon^2$. There is an $\epsilon_0$, such that $\forall \epsilon < \epsilon_0$, $v_k \neq 0$. Also $|\langle v_1,  v_2 \rangle| \leq |\langle u_1, u_2 \rangle| + 2 \epsilon + \epsilon^2$ and $|v_1||v_2| \geq 1- \epsilon^2$. So if $1-\epsilon^2 > |\langle u_1, u_2 \rangle| + 2 \epsilon + \epsilon^2$, then $|\langle v_1,  v_2 \rangle| < |v_1| |v_2|$. There is an $\epsilon_0$, such that if $\epsilon < \epsilon_0$, this inequality holds. We showed that $\{v_1,v_2\}$ is a basis of $N$, and $N = span(\{v_1,v_2\})$.\\
We do the same computation for $\mathbb{R}_0^-$, and obtain the vectors $v_3$ and $v_4$.\\
Let $E(\lambda)$ be the Evans function for \eqref{system_constant_coefficients1}. Then $E(\lambda) = \det([u_1,u_2,u_3,u_4])$. Let $E_p(\lambda)$ be the Evans function for \eqref{perturbed_system}. Then $E_p(\lambda) = \det([v_1,v_2,v_3,v_4])$. Denote $H = [h_1,...,h_4]$, $U = [u_1,...,u_4]$. Let $\|A\|_F = \sqrt{\sum_{j,k} |a_{j,k}|^2}$. The inequality $\|A\|_2 \leq \|A\|_F$ holds. 
Suppose $U$ is the identity matrix. Let $q$ be an eigenvalue of $U+H$. The Bauer-Fike theorem implies that
\begin{equation*}
|1-q|\leq\|H\|_2.
\end{equation*}
Since $\|H\|_2 \leq \|H\|_F$, $\|H\|_F<1$ shows that $0$ cannot be an eigenvalue of $U+H$, hence $det(U+H)\neq 0$. Therefore $E_p(\lambda)\neq 0$.\\
Let $s_j$, $j=1,...,4$ be the eigenvectors of $A(\tilde{\lambda})$ and $S = [s_1,s_2,s_3,s_4]$. Let $D(\tilde{\lambda}) = diag(z_1,z_2,z_3,z_4) = S^{-1} A(\tilde{\lambda})S $. We make the change of variable $u = S v$. Then \eqref{perturbed_system} becomes
\begin{equation*}
\frac{dv}{dx} = D(\tilde{\lambda})v + S^{-1}B(x,\lambda)S v.
\end{equation*}
Since $D(\tilde{\lambda})$ is diagonal its eigenvectors are $u_j = e_j$, the standard basis vectors, for $j=1,...,4$. Hence $U = Id$, $\det(U) = \|U\|_2=\|U^{-1}\|_2=1$, and $\kappa = 1$. Moreover $P = diag(1,1,0,0)$ and $K = 1$.\\
We suppose for simplicity that $k=\sqrt{2}$, although all the calculations can be done for arbitrary value of $k$.
We have $0<\theta_1<\pi/2$, $-\pi/2 < \theta_2 < 0$ and $\theta_2 = - \theta_1$. Let $\tilde{\lambda} = \exp(i \theta)$. Then
\begin{equation}
\label{ineq_arg}
-\frac{\pi}{4}+\frac{\theta_j}{2} \leq \frac{\theta + \theta_j}{2}\leq \frac{\pi}{4}+\frac{\theta_j}{2},\mbox{ }j=1,2.
\end{equation}
The roots of \eqref{quadratic_char_eq} are $z_k$. We have $\alpha = \min_k \Re(z_k)$. From \eqref{ineq_arg} it follows that we can take $\alpha = \Re(exp(i(\theta+\theta_1)/2))$ for $\theta=\pi/2$. Since $\alpha$ depends on $\theta$, we can get a non-ciruclar region where the eigenvalues are contained by computing $\alpha$ for different $\theta$.\\
We can directly compute $S^{-1}B(x,\lambda)S$. Since $B(x,\lambda)$ has all elements except the last row equal to zero, that is $b_{4,k}$, we obtain 
\begin{eqnarray*}
\|S^{-1}B(x,\lambda)S\|_2 \leq \|S^{-1}B(x,\lambda)S\|_F \leq \|R\|_F
\end{eqnarray*}
and for the matrix  $R$ from \eqref{residual_matrix},
where $m_k(\lambda)$ are some upper bounds for $b_{4,k}(\lambda)$, $|b_{4,k}(x,\lambda)|\leq m_k(\lambda)$.
For equation \eqref{rescaled_equation_nonlinear_dispresion} we can choose $m_k(\lambda)$ from \eqref{bounds}.
Note that the matrix $R(\lambda)$ from \eqref{residual_matrix} has monotonically decreasing with $|\lambda|$ elements. We get two matrices $R_{\pm}$ corresponding to $\mathbb{R}_0^+$ and $\mathbb{R}_0^-$ respectively, and two values for $\delta_{\pm}$ and $\epsilon_{\pm}$. Moreover $\|H\|_F \leq \sqrt{2} \sqrt{\epsilon_+^2 + \epsilon_-^2}$. \end{proof}
For the set of parameters \eqref{parameters}, to which Lemma \ref{lemma_global_existence} applies, we can show numerically using Lemma \ref{upper_bound_lambda} that there are no eigenvalues with $\Re(\lambda)\geq 0$ and $|\lambda| \geq 1.9 \cdot 10^4$. By Evans function computations we can check numerically that there are no eigenvalues with $\Re(\lambda)\geq 0$ and $|\lambda| < 1.9 \cdot 10^4$ (see \cite{LMZ18_num}). This is numerical evidence of point spectrum stability.

\end{document}